\documentclass{article}
\usepackage{proof,prooftree}
\usepackage{amsfonts,amssymb}
\usepackage{amsthm}
\usepackage{latexsym}
\usepackage{pslatex}
\usepackage{stmaryrd}
\usepackage{fancybox}
\usepackage{amsmath} %
\usepackage{amssymb,prooftree,proof}
\usepackage{latexsym}
\usepackage[all]{xy}
\usepackage{qsymbols}
\usepackage{pslatex}
\usepackage{graphicx}
\usepackage{authblk}

\newenvironment{keyword}[1]{\medskip\noindent\textbf{Keywords:}\ #1}

\newcount\timehh\newcount\timemm \timehh=\time
\divide\timehh by 60 \timemm=\time
\newcommand{\timestamp}{
{\protect\small\sl\today\ --
  \ifnum\timehh<10 0\fi\number\timehh\,:\,
  \ifnum\timemm<10 0\fi\number\timemm}}

\newif\ifcomment

\usepackage[usenames,dvipsnames]{color}     

\newcommand{\jel}[1]{\ifcomment {\color{cyan} {#1}} \else #1\fi}
\definecolor{darkbrown}{cmyk}{.3,.75,.75,.15}
\newcommand{\pierre}[1]{\ifcomment {\color{darkbrown} {#1}} \else #1\fi}

\newcommand{\via}[1]{}
\newcommand{\SKIP}[1]{}

\newcommand{\mmu}{\widetilde{\mu}}

\newcommand{\lmm}{\ensuremath{\overline{\lambda} \mu\mmu}}

\newcommand{\lG}{\ensuremath{\lambda^{\mathsf{Gtz}}}}

\newcommand{\LlG}{`L_{\circledR}^{\mathsf{Gtz}}}
\newcommand{\LlGC}{`L_{\circledR,C}^{\mathsf{Gtz}}}
\newcommand{\llG}{\lambda_{\circledR}^{\mathsf{Gtz}}}
\newcommand{\Rcl}{\Lambda_{\circledR}}
\newcommand{\rcl}{\lambda_{\circledR}}
\newcommand{\rc}{\circledR}
\newcommand{\llxr}{\lambda \mathsf{lxr}}
\newcommand{\LGvdash}[2]{\ensuremath{\Rcl{\scriptscriptstyle(#1\vdash_{\rcl} #2)}}}

\newcommand{\isub}[2]{[#1/#2]}   
\newcommand{\app}[2]{#1@#2}

\newcommand{\bindx}{\widehat x.}
\newcommand{\bindy}{\widehat y.}

\newcommand{\append}[2]{#1@#2}

\newcommand{\weak}[2]{#1 \odot #2}
\newcommand{\cont}[4]{#1 < ^{#2}_{#3}#4}

\newcommand{\gcw}{\ensuremath{\llG}}

\newcommand{\size}[1]{\mathcal{S}(#1)}
\newcommand{\cnorm}[1]{|\!|#1|\!|_\mathsf{c}}
\newcommand{\wnorm}[1]{|\!|#1|\!|_\mathsf{w}}

\newcommand{\isubs}[2]{#1/#2}   
\newcommand{\tlam}{{Types}} 
\newcommand{\lefti}{[ \! [}      
\newcommand{\righti}{] \! ]}
\newcommand{\ti}[1]{\lefti #1 \righti}
\newcommand{\tei}{\ti} 
\newcommand{\SN}{\ensuremath{\mathcal{SN}}} 

\newcommand{\VAR}{\textsf{VAR}}
\newcommand{\RED}{\textsf{RED}}
\newcommand{\SAT}{\textsf{EXP}}  
\newcommand{\WEAK}{\textsf{WEAK}}
\newcommand{\CONT}{\textsf{CONT}}

\newcommand{\rdx}{\textit{redex}}
\newcommand{\ctr}{\textit{contr}}

\newcommand{\vX}{\mathcal{X}} 
\newcommand{\vM}{\mathcal{M}} 
\newcommand{\vN}{\mathcal{N}} 
\newcommand{\tA}{\alpha} 
\newcommand{\tB}{\beta} 
\newcommand{\tC}{\gamma} 
\newcommand{\tS}{\sigma} 
\newcommand{\tT}{\tau} 
\newcommand{\tR}{\rho} 
\newcommand{\tU}{\upsilon}

\newcommand{\fsto}{\xymatrix@C=15pt{\ar@{>}[r] &}}

\newcommand{\f}{f}
\newcommand{\g}{g}

\newcommand{\lcw}{\lambda_{\mathsf{c}\mathsf{w}}}
\newcommand{\LR}{\Lambda_{\circledR}}

\newcommand{\LRvdash}[2]{\ensuremath{\Rcl{\scriptscriptstyle(#1\vdash_{\rcl} #2)}}}

\newcommand{\xsub}[2]{\langle #2 := #1 \rangle}   
\newcommand{\Rclx}{\Lambda^{\mathsf{x}}_{\circledR}}
\newcommand{\rclx}{\lambda^{\mathsf{x}}_{\circledR}}
\newcommand{\intt}[1]{\lfloor #1 \rfloor}
\newcommand{\intc}[1]{\lfloor #1 \rfloor_\mathsf{k}}
\newcommand{\dztop}{\Delta_{0}^{\top}}
\newcommand{\gztop}{\Gamma_{0}^{\top}}
\newcommand{\gtop}{\Gamma^{\top}}

\theoremstyle{plain}
\newtheorem{theorem}{Theorem}
\newtheorem{proposition}[theorem]{Proposition}
\newtheorem{lemma}[theorem]{Lemma}
\newtheorem{example}[theorem]{Example}

\theoremstyle{definition}
\newtheorem{definition}[theorem]{Definition}

\newcommand{\ead}[1]{\thanks{\textsf{Email:}~#1}}

\newif\ifcomment
\commentfalse           

\begin{document}

\title{A journey through resource control lambda calculi and explicit substitution
  using intersection types}
\author[1]{S. Ghilezan \ead{gsilvia@uns.ac.rs}}
\author[1]{J. Iveti\' c  \ead{jelenaivetic@uns.ac.rs}}
\author[2]{P. Lescanne \ead{pierre.lescanne@ens-lyon.fr}}
\author[3]{S. Likavec \ead{likavec@di.unito.it}}
\affil[1]{University of Novi Sad, Faculty of Technical Sciences,  Serbia}
\affil[2]{University of Lyon, \' Ecole Normal Sup\' erieure de Lyon, France}
\affil[3]{Dipartimento di Informatica, Universit\`a di Torino, Italy}
 \date{}

\maketitle

\begin{abstract}
In this paper we invite the reader to a journey through three
lambda calculi with resource control: the lambda calculus, the
sequent lambda calculus, and the lambda calculus with explicit
substitution. All three calculi enable explicit control of
resources due to the presence of weakening and contraction
operators. Along this journey, we propose intersection type
assignment systems for all three resource control calculi. We
recognise the need for three kinds of variables all requiring
different kinds of intersection types. Our main contribution is
the characterisation of strong normalisation of reductions in all
three calculi, using the techniques of reducibility, head subject
expansion, a combination of well-orders and suitable embeddings of
terms.

\end{abstract}

\section*{Introduction}
\label{sec:intro}

It is well known that simply typed $\lambda$-calculus captures the
computational content of intuitionistic natural deduction through
Curry-Howard correspondence~\cite{howa80}. This connection between
logic and computation can be extended to other calculi and logical
systems~\cite{ghillika08}: Herbelin's
$\overline{\lambda}$-calculus~\cite{herb95b}, Pinto and Dyckhoff's
$\lambda\pi\sigma$-calculus~\cite{pintdyck98} and Esp\'{\i}rito
Santo's $\lG$-calculus~\cite{jesTLCA07} correspond to
intuitionistic sequent calculus. In the realm of classical logic,
Parigot's $\lambda\mu$-calculus~\cite{pari92} corresponds to
classical natural deduction,  whereas Barbanera and Berardi's
symmetric calculus~\cite{barbbera96} and Curien and Herbelin's
$\lmm$-calculus~\cite{curiherb00} correspond to its sequent
version. Extending first, the $\lambda \mathsf{x}$~calculus of
explicit substitution and later $\lambda$-calculus and
$\lG$-calculus with explicit operators for erasure (a.k.a.\
weakening) and duplication (a.k.a.\ contraction) brings the same
correspondence to intuitionistic natural deduction and
intuitionistic sequent calculus with explicit structural rules of
weakening and contraction on the logical side~\cite{gentzen35}, as
investigated in~\cite{kesnleng07,kesnrena09,ghilivetlesczuni11}.

On the other hand, let us consider type assignment systems for
various calculi. To overcome the limitations of the simple type
discipline in which the only forming operator is an
arrow~$\rightarrow$, a new type forming operator $\cap$ was
introduced in~\cite{coppdeza78,coppdeza80,pott80,sall78}. The
newly obtained intersection type assignment systems enabled
complete characterisation of  termination of term
calculi~\cite{bake92,gall98,ghil96}. The extension of Curry-Howard
correspondence to other formalisms brought the need for
intersection types into many different
settings~\cite{dougghillesc07,kikuchiRTA07,matthes2000,neer05}.

Our work is inspired by and extends Kesner and
Lengrand's~\cite{kesnleng07} work on resource operators for
$\lambda$-calculus with explicit substitution. Their linear
$\llxr$-calculus introduces operators for linear substitution,
erasure and duplication, preserving at the same time confluence
and full composition of explicit substitutions of its predecessor
${\lambda \mathsf{x}}$~\cite{BlooRose95,RoseBlooLang:jar2011}. The
simply typed version of this calculus corresponds to the
intuitionistic fragment of Linear Logic's proof-nets, according to
Curry-Howard correspondence, and it enjoys strong normalisation
and subject reduction. Resource control in sequent
$\lambda$-calculus was proposed by Ghilezan et al. in
\cite{ghilivetlesczuni11}, whereas resource control both in
$\lambda$-calculus and  ${\lambda \mathsf{x}}$-calculus was
further developed in \cite{kesnrena09,kesnrena11}.

In order to control all resources, in the  spirit of $\lambda
I$-calculus (see e.g. \cite{bare84}), void lambda abstraction is
not acceptable, so in order to have $\lambda x.M$ the variable $x$
has to occur in~$M$. But if $x$ is not used in a term $M$, one can
perform an \emph{erasure (a.k.a weakening)} by using the
expression $\weak{x}{M}$. In this way, the term $M$ does not
contain the variable $x$, but the term $\weak{x}{M}$ does.
Similarly, a variable should not occur twice. If nevertheless, we
want to have two positions for the same variable, we have to
duplicate it explicitly, using fresh names. This is done by using
the operator $\cont{x}{x_1}{x_2}{M}$, called \emph{duplication
(a.k.a contraction)} which creates two fresh variables $x_1$ and
$x_2$.

Explicit control of erasure and duplication leads to decomposing
of reduction steps into more atomic steps, thus revealing the
details of computation which are usually left implicit.  Since
erasing and duplicating of (sub)terms essentially changes the
structure of a program, it is important to see how this mechanism
really works and to be able to control this part of computation.
We chose a direct approach to term calculi rather than taking a
more common path through linear logic \cite{abra93,bent93}. In
practice, for instance in the description of compilers by rules
with
binders~\cite{rose:LIPIcs:2011:3130,rose11:_implem_trick_that_make_crsx_tick},
the implementation of substitutions of linear variables by
inlining\footnote{\emph{Inlining} is the technics which
consists in copying at compile time
    the text of a function instead of implementing a call to that function.} is simple and efficient when substitution of duplicated
variables requires the cumbersome and time consuming mechanism of
pointers and it is therefore important to tightly control
duplication.  On the other hand, precise control of erasing does
not require a garbage collector and prevents memory leaking.

Our main goal is to characterize the termination of reductions for
term calculi with explicit control of duplication and erasure, in
different frameworks: natural deduction, sequent style and with
explicit substitution. We revisit the syntax and reduction rules
of three term calculi with explicit operators for  weakening and
contraction: $\rcl$ (the extension of the $\lambda$-calculus),
$\llG$ (the extension of the sequent lambda calculus $\lG$) and
$\rclx$ (the extension of the $\lambda \mathsf{x}$-calculus with
explicit substitution). We then introduce intersection types into
all three calculi $\rcl$, $\llG$ and $\rclx$. Our intersection
type assignment systems $\rcl\cap$, $\llG\cap$ and $\rclx\cap$
integrate intersection into logical rules, thus preserving
syntax-directedness of the system. We assign restricted form of
intersection types to terms, namely strict types, therefore
minimizing the need for pre-order on types. By using these
intersection type assignment systems we prove that terms in all
three calculi enjoy strong normalisation if and only if they are
typeable. To the
  best of our knowledge, together with the conference version of this paper
  \cite{ghilivetlikalesc11}, this is the first treatment of intersection types in the presence of
  resource control operators. Intersection types fit naturally to resource control. Indeed, the
  control allows us to consider three types of variables: variables as placeholders (the
  traditional view of $`l$-calculus), variables to be duplicated and variables
    to be erased because they are irrelevant. For each kind of a variable, there is a kind of type
  associated to it, namely a strict type for a \emph{placeholder}, an intersection for a variable
  \emph{to-be-duplicated}, and a specific type for an \emph{erased} variable.

We first prove in Section~\ref{sec:lambda} that  terms typeable in
$\rcl$-calculus are strongly normalising
 by adapting the reducibility method for explicit resource control operators.
Then we prove that all strongly normalising terms are typeable in
$\rcl$-calculus by using typeability of normal forms and head
subject expansion.

Further, we prove strong normalisation for $\llG$ and $\rclx$, in
Section~\ref{sec:lambda_gtz} and Section~\ref{sec:lambda_xsub},
respectively, by using a combination of well-orders and a suitable
embeddings of $\llG$-terms and $\rclx$-terms into $\rcl$-terms
which preserve typeability and enable the simulation of all
reductions and equations by the operational semantics of the
$\rcl$-calculus. Finally, we prove that strong normalisation
implies typeability in $\llG$ and $\rclx$ using head subject
expansion.

\paragraph{Related work}

The idea to control the use of variables can be traced  back to Church's $\lambda I$-calculus~\cite{bare84}.
Currently there are several different lines of research in resource aware term calculi.
Van Oostrom~\cite{oost01} and later Kesner and Lengrand~\cite{kesnleng07}, applying ideas from linear logic~\cite{LL},
proposed to extend the $\lambda$-calculus and the $\lambda \mathsf{x}$-calculus, with operators to control the use of
variables (resources). Generalising this approach, Kesner and Renaud~\cite{kesnrena09,kesnrena11}
developed the {\em prismoid of resources}, a system of eight calculi parametric over the explicit
and implicit treatment of substitution, erasure and duplication. Resource control in sequent
calculus corresponding to classical logic was proposed in~\cite{zunicPHD}.
On the other hand, process calculi and their relation to $\lambda$-calculus by Boudol~\cite{boud93} initialised
investigations in resource aware non-deterministic $\lambda$-calculus with multiplicities and a generalised notion of
application~\cite{boudcurilava99}. The theory was connected to linear logic via differential
$\lambda$-calculus in~\cite{ehrhregn03} and typed with non-idempotent intersection types in~\cite{pagaronc10}.
In this paper we follow the notation of~\cite{zunicPHD} and~\cite{ghilivetlikalesc11}, which is related to~\cite{oost01}.

This paper is an extended and revised version of
\cite{ghilivetlikalesc11}. In addition to $\rcl$-calculus and
$\llG$-calculus presented in \cite{ghilivetlikalesc11}, this
extended version adds the treatment of the $\rclx$-calculus, the
resource lambda calculus with explicit substitution, together with
the characterization of strong normalisation for this calculus.
Also, the proof that typeability implies strong normalisation in
$\rcl$-calculus is improved.

\paragraph{Outline of the paper}

In Section~\ref{sec:lambda} we first give the syntax and reduction
rules for $\rcl$-calculus, followed by the intersection type
assignment system and the characterisation of strong
normalisation. Section~\ref{sec:lambda_gtz} deals with
$\llG$-calculus, its syntax, reduction rules, intersection type
assignment system and the characterisation of strong
normalisation. Section~\ref{sec:lambda_xsub} introduces
$\rclx$-calculus with its syntax, reduction rules and intersection
type assignment system, again followed by the characterisation of
strong normalisation. Finally, we conclude in
Section~\ref{sec:conclusion} with some directions for future work.

\tableofcontents

\section{Intersection types for the resource control lambda calculus $\rcl$}
\label{sec:lambda}

In this section we focus on the resource control lambda calculus
$\rcl$. First we revisit its syntax and operational semantics;
further we introduce intersection type assignment system and
finally we prove that typebility in the proposed system completely
characterises the set of strongly normalising $\rcl$-terms.

\subsection{Resource control lambda calculus $\rcl$}
\label{subsec:rcl}

The \emph{resource control} lambda calculus, $\rcl$, is an
extension of the $\lambda$-calculus with explicit operators for
weakening and contraction. It corresponds to the $\lcw$-calculus
of Kesner and Renaud, proposed in \cite{kesnrena09} as a vertex of
``the prismoid of resources", where substitution is implicit. We
use a notation along the lines of \cite{zunicPHD} and close to
\cite{oost01}. It is slightly modified w.r.t.~\cite{kesnrena09} in
order to emphasize the correspondence between this calculus and
its sequent counterpart.

First of all, we introduce the syntactic category of {\em pre-terms} of $\rcl$-calculus given by the following
abstract syntax:
$$
\begin{array}{lcrcl}
\textrm{Pre-terms}    &  & \f & ::= & x \,|\,\lambda x.\f \,|\,
\f \f\,|\,\weak{x}{\f} \,|\, \cont{x}{x_1}{x_2}{\f}
\end{array}
$$
where $x$ ranges over a denumerable set of term variables.
$\lambda x.\f$ is an \emph{abstraction}, $\f\f$ is an \emph{application},
$\weak{x}{\f}$ is a \emph{weakening} and $\cont{x}{x_1}{x_2}{\f}$
is a \emph{contraction}. The contraction operator is assumed to be
insensitive to the order of the arguments $x_1$ and~$x_2$, i.e.
$\cont{x}{x_1}{x_2}{\f}=\cont{x}{x_2}{x_1}{\f}$.

The set of free variables of a pre-term $\f$, denoted by $Fv(\f)$,
is defined as follows:

\hspace*{10mm}$\begin{array}{c}
Fv(x) = x;\quad Fv(\lambda x.\f) = Fv(\f)\setminus\{x\}; \quad
Fv(\f \g) = Fv(\f) \cup Fv(\g);\\
Fv(\weak{x}{\f}) = \{x\} \cup Fv(\f); \quad Fv(\cont{x}{x_1}{x_2}{\f}) = \{x\} \cup Fv(\f)\setminus
\{x_1,x_2\}.
\end{array}$

In $\cont{x}{x_1}{x_2}{\f}$, the contraction binds the variables
$x_1$ and $x_2$ in $f$ and introduces a free variable $x$. The
operator $\weak {x}{\f}$ also introduces a free variable $x$. In
order to avoid parentheses, we let the scope of all binders extend
to the right as much as possible.

The set of $\rcl$-{\em terms}, denoted by $\LR$ and ranged over by $M,N,P,M_1,...$. is a subset of
the set of pre-terms, defined in Figure~\ref{fig:wf}.
\begin{figure}[htpb]
\centerline{ \framebox{ $
    \begin{array}{c}
      \begin{array}{c@{\qquad\qquad}c}
        \infer{x \in \LR}{} &
        \infer{\lambda x.\f \in \LR}
        {\f \in \LR \;\; x \in Fv(\f)}
      \end{array}
      \\\\
    \infer{\f \g \in \LR}
    {\f \in \LR\;\; \g \in \LR \;\; Fv(\f) \cap Fv(\g) = \emptyset}
    \\\\
    \begin{array}{c@{\qquad}c}
      \infer{\weak{x}{\f} \in \LR}
      {\f \in \LR \;\; x \notin Fv(\f)}&
      \infer{\cont{x}{x_1}{x_2}{\f} \in \LR}
      {\f \in \LR \;\; x_{1} \not = x_{2}\;\; x_1,x_2 \in Fv(\f) \;\; x \notin Fv(\f) \setminus \{x_{1},x_{2}\}}
    \end{array}
  \end{array}
$ }} \caption{$\LR$:  $\rcl$-terms}
\label{fig:wf}
\end{figure}

Informally, we say that a term is a pre-term in which in every
subterm every free variable occurs exactly once, and every binder
binds (exactly one occurrence of) a free variable. Our notion of
terms corresponds to the notion of linear terms in
\cite{kesnleng07}. In that sense, only linear expressions are in
the focus of our investigation. In other words, terms are
well-formed in $\rcl$ if and
  only if bound variables appear actually in the term and
  variables occur at most once. These conditions will be assumed
  throughout the paper without mentioning them explicitly.
This assumption is not a restriction, since every traditional term
has a  corresponding $\rcl$-term, as illustrated by the following
example.

\begin{example}
  Pre-terms $\lambda x.y$ and $\lambda x. xx$ are not $\rcl$-terms, on
  the other hand pre-terms $\lambda x.(\weak{x}{y})$ and $\lambda
  x.\cont{x}{x_1}{x_2}(x_1 x_2)$ are their corresponding $\rcl$-terms.
\end{example}

In the sequel, we use the notation $\weak{X}{M}$ for
$\weak{x_1}{...\;\weak{x_n}{M}}$ and $\cont{X}{Y}{Z}{M}$ for
$\cont{x_1}{y_1}{z_1}{...\;\cont{x_n}{y_n}{z_n}{M}}$, where $X$,
$Y$ and $Z$ are lists of size $n$, consisting of all distinct
variables $x_1,...,x_n, y_1,..., y_n, z_1,...,z_n$.
If $n=0$, i.e.,  if $X$ is the empty list, then $\weak{X}{M} = \cont{X}{Y}{Z}{M} = M$.
Note that due to the equivalence relation defined in Figure~\ref{fig:equiv-rcl},
we can use these notations also for sets of variables of the same size.

In what follows we use Barendregt's convention~\cite{bare84} for variables:
in the same context a variable cannot be both free and bound. This applies to
binders like $`l x.M$ which binds $x$ in $M$, $\cont{x}{x_1}{x_2}M$ which binds $x_1$ and
$x_2$ in $M$, and also to the implicit substitution $M[N/x]$ which can be seen as a binder for $x$ in $M$.

The set $\circledR$ of  reduction rules
$\rightarrow_{\rcl}$ of the $\rcl$-calculus is presented in
Figure~\ref{fig:red-rcl}.
\begin{figure}[htbp]
\centerline{ \framebox{ $
\begin{array}{rrcl}
(\beta)             & (\lambda x.M)N & \rightarrow & M\isub{N}{x} \\[1mm]
(\gamma_1)          & \cont{x}{x_1}{x_2}{(\lambda y.M)} & \rightarrow & \lambda y.\cont{x}{x_1}{x_2}{M} \\
(\gamma_2)          & \cont{x}{x_1}{x_2}{(MN)} & \rightarrow & (\cont{x}{x_1}{x_2}{M})N, \;\mbox{if} \; x_1,x_2 \not\in Fv(N)\\
(\gamma_3)          & \cont{x}{x_1}{x_2}{(MN)} & \rightarrow & M(\cont{x}{x_1}{x_2}{N}), \;\mbox{if} \; x_1,x_2 \not\in Fv(M)\\[1mm]
(\omega_1)          & \lambda x.(\weak{y}{M}) & \rightarrow & \weak{y}{(\lambda x.M)},\;x \neq y\\
(\omega_2)          & (\weak{x}{M})N & \rightarrow & \weak{x}{(MN)}\\
(\omega_3)          & M(\weak{x}{N}) & \rightarrow & \weak{x}{(MN)}\\[1mm]
(\gamma \omega_1)   & \cont{x}{x_1}{x_2}{(\weak{y}{M})} & \rightarrow & \weak{y}{(\cont{x}{x_1}{x_2}{M})},\;y \neq x_1,x_2\\
(\gamma \omega_2)   & \cont{x}{x_1}{x_2}{(\weak{x_1}{M})} & \rightarrow & M \isub{x}{x_2}\\
\end{array}
$ }} \caption{The set $\circledR$ of reduction rules of the
$\rcl$-calculus} \label{fig:red-rcl}
\end{figure}

The reduction rules are divided into four groups. The main
computational step is $\beta$ reduction. The group of $(\gamma)$
reductions perform propagation of contraction into the expression.
Similarly, $(\omega)$ reductions extract weakening out of
expressions. This discipline allows us to optimize the computation
by delaying duplication of terms on the one hand, and by
performing erasure of terms as soon as possible on the other.
Finally, the rules in $(\gamma\omega)$ group explain the
interaction between explicit resource operators that are of
different nature.

The inductive definition of the meta operator $\isub{\;}{\;}$,
representing the implicit substitution of free variables, is given
in Figure~\ref{fig:sub-rcl}. In order to obtain well formed terms
as the results of substitution, $Fv(M) \cap Fv(N) = \emptyset$
must hold in this definition. Moreover, notice that for the
expression $M\isub{N}{x}$ to make sense, $M$ must contain exactly
one occurrence of the free variable $x$ and $M$ and $N$ must share
no variable but $x$.\footnote{We prefer $x$ not to belong to $M$
in order to respect Barendregt convention on variable.} Indeed a
substitution is always created by a $\beta$-reduction and, in the
term $(\lambda x.M)N$, $x$ has to appear exactly once in $M$ and
the other variables of $Fv(M) \cup Fv(N)$ as well. Barendregt
convention on variable says that $x$ should not occur freely in
$N$.  Also, if the terms $N_1$ and $N_2$ are obtained from the
term $N$ by renaming all the free variables in $N$ by fresh
variables, then $M[N_1/x_1,N_2/x_2]$ denotes a parallel
substitution.

\begin{figure}[ht]
\centerline{ \framebox{ $
\begin{array}{rcl}
x\isub{N}{x} & \triangleq & N \\
(\lambda y.M)\isub{N}{x} & \triangleq  & \lambda y.M\isub{N}{x},\;\;x \neq y \\
(MP)\isub{N}{x} & \triangleq  & M\isub{N}{x} P, \;\;x \not\in Fv(P) \\
(MP)\isub{N}{x} & \triangleq  & M P\isub{N}{x}, \;\;x \not\in Fv(M) \\
(\weak{y}{M})\isub{N}{x} &\triangleq  & \weak{y}{M\isub{N}{x}}, \;\;x \neq y\\
(\weak{x}{M})\isub{N}{x} & \triangleq  & \weak{Fv(N)}{M}\\
(\cont{y}{y_1}{y_2}{M})\isub{N}{x} & \triangleq  &
\cont{y}{y_1}{y_2}{M\isub{N}{x}}, \;\;x \neq y\\
(\cont{x}{x_1}{x_2}{M})\isub{N}{x} & \triangleq  &
\cont{Fv(N)}{Fv(N_1)}{Fv(N_2)}{M[N_1/x_1,N_2/x_2]}
\end{array}
$ }} \caption{Substitution in $\rcl$-calculus}
\label{fig:sub-rcl}
\end{figure}

\begin{definition}[Parallel substitution] \label{def:par-sub}
 $M[N/x,P/z] = (M\isub{N}{x})\isub{P}{z}$ for $\; x,z \in
    Fv(M)$ and
    $\;(Fv(M)\setminus\{x\}) \cap Fv(N) = (Fv(M)\setminus\{z\}) \cap Fv(P) = Fv(N) \cap Fv(P) =
    \emptyset$.
\end{definition}

In the $\rcl$-calculus, one works modulo equivalencies given in
Figure~\ref{fig:equiv-rcl}.

\begin{figure}
\centerline{ \framebox{ $
\begin{array}{lrcl}
(`e_1) & \weak{x}{(\weak{y}{M})} & \equiv_{\rcl} &
\weak{y}{(\weak{x}{M})}\\
(`e_2) & \cont{x}{x_1}{x_2}{M} & \equiv_{\rcl} & \cont{x}{x_2}{x_1}{M}\\
(`e_3) & \cont{x}{y}{z}{(\cont{y}{u}{v}{M})} & \equiv_{\rcl} &
\cont{x}{y}{u}{(\cont{y}{z}{v}{M})}\\
(`e_4) & \cont{x}{x_1}{x_2}{(\cont{y}{y_1}{y_2}{M})} &
\equiv_{\rcl}  & \cont{y}{y_1}{y_2}{(\cont{x}{x_1}{x_2}{M})},\;\;
x \neq y_1,y_2, \; y \neq x_1,x_2
\end{array}
$ }} \caption{Equivalences in $\rcl$-calculus}
\label{fig:equiv-rcl}
\end{figure}

Notice that because we work with $\rcl$ terms, no variable is lost
during the computation, which is stated by the following
proposition.

\begin{proposition}
\label{prop:pres-of-FV} If $M \to M'$ then $Fv(M) = Fv(M').$
\end{proposition}
\begin{proof}
The proof is by case analysis on the reduction rules.
\end{proof}

The following lemma explains how to compose implicit
substitutions.

\begin{lemma}
\label{lemma:subs-comp} \rule{0in}{0in}
\begin{itemize}
\item If  $z`:FV(N)$ then $(M\isub{N}{x})\isub{P}{z} =
  M\isub{N\isub{P}{z}}{x}$.
\item If $z`:FV(M)$ then $(M\isub{N}{x})\isub{P}{z} =
  (M\isub{P}{z})\isub{N}{x} $
\end{itemize}
\end{lemma}

\begin{proof}
\rule{0in}{0in}

Notice that for the expressions to make sense, one must have
  $x`:Fv(M)$  and
  $(Fv(M)\setminus\{x\}) \cap Fv(N) = \emptyset$,  %
  $ (Fv(N)\setminus\{z\}) \cap   Fv(P)= \emptyset$ and %
  ${(Fv(M) \setminus\{x\}) \cap Fv(P)  = \emptyset}$. %

\begin{itemize}

\item
$(x\isub{N}{x})\isub{P}{z}  \triangleq  N\isub{P}{z} \mbox{ and }
x\isub{N\isub{P}{z}}{x} \triangleq N\isub{P}{z}$

\item
$((\lambda y.M)\isub{N}{x})\isub{P}{z}  \triangleq (\lambda y.M\isub{N}{x})\isub{P}{z} \triangleq \lambda y.(M\isub{N}{x})\isub{P}{z}  =_{IH} \lambda y.M\isub{N\isub{P}{z}}{x} \triangleq (\lambda y.M)\isub{N\isub{P}{z}}{x} ,\;\;x,z \neq y$

\item
$x \not\in Fv(Q)$ (the case $x \not\in Fv(M)$ is analogous)\\
$((MQ)\isub{N}{x})\isub{P}{z}  \triangleq
(M\isub{N}{x}Q)\isub{P}{z} =_{z \in Fv(N)}  (M\isub{N}{x})\isub{P}{z}Q =_{IH}
M\isub{N\isub{P}{z}}{x}Q = (MQ)\isub{N\isub{P}{z}}{x}$

\item
$((\weak{y}{M})\isub{N}{x})\isub{P}{z} \triangleq (\weak{y}{M\isub{N}{x}})\isub{P}{z} \triangleq \weak{y}{M\isub{N}{x}}\isub{P}{z} =_{IH} \weak{y}{M\isub{N\isub{P}{z}}{x}} \triangleq (\weak{y}{M})\isub{N\isub{P}{z}}{x}, \;\;y \neq x,z$

\item
$((\weak{x}{M})\isub{N}{x})\isub{P}{z}  \triangleq   (\weak{Fv(N)}{M})\isub{P}{z} =_{z \in Fv(N)}  \\
(\weak{z}{\weak{\{Fv(N) \setminus \{z\}\}}{M}})\isub{P}{z} =
\weak{\{Fv(P) \cup Fv(N) \setminus z\}}{M} = \weak{Fv(N\isub{P}{z})}{M}  \triangleq
(\weak{x}{M})\isub{N\isub{P}{z}}{x}$

\item $((\cont{y}{y_1}{y_2}{M})\isub{N}{x})\isub{P}{z}
\triangleq (\cont{y}{y_1}{y_2}{M\isub{N}{x}})\isub{P}{z}
\triangleq \cont{y}{y_1}{y_2}{M\isub{N}{x}}\isub{P}{z} =_{IH}
\cont{y}{y_1}{y_2}{M\isub{\isub{P}{z}N}{x}} =_{IH}
(\cont{y}{y_1}{y_2}{M})\isub{\isub{P}{z}N}{x}, \;\;x \neq y$

\item
$((\cont{x}{x_1}{x_2}{M})\isub{N}{x})\isub{P}{z}  \triangleq
(\cont{Fv(N)}{Fv(N_1)}{Fv(N_2)}{M[N_1/x_1,N_2/x_2]})\isub{P}{z} \triangleq\\
\cont{Fv(N)}{Fv(N_1)}{Fv(N_2)}{M\isub{N_1}{x_1}\isub{N_2}{x_2}}\isub{P}{z}
= \\
\cont{z}{z_1}{z_2}{}\cont{Fv(N)\setminus
\{z\}}{Fv(N_1)\setminus \{z_1\}}{Fv(N_2)\setminus
\{z_2\}}{M\isub{N_1}{x_1}\isub{N_2}{x_2}}\isub{P}{z} \triangleq\\
\cont{Fv(P)}{Fv(P_1)}{Fv(P_2)}{}\cont{Fv(N)\setminus
\{z\}}{Fv(N_1)\setminus \{z_1\}}{Fv(N_2)\setminus
\{z_2\}}{M\isub{N_1}{x_1}\isub{N_2}{x_2}
\isub{P_1}{z_1}\isub{P_2}{z_2}}=_{IH}\\
\cont{Fv(P) \cup Fv(N)\setminus \{z\}} {Fv(P_1) \cup
Fv(N_1)\setminus \{z_1\}}{Fv(P_2) \cup Fv(N_2)\setminus
\{z_2\}}{M[N_1\isub{P_{1}}{z_{1}}/x_1,N_2\isub{P_{2}}{z_{2}}/x_2]}
\triangleq\\
(\cont{x}{x_1}{x_2}{M})\isub{N\isub{P}{z}}{x}.$\\
We used the fact that $z_1 \in Fv(N_1)$ and $z_2 \in Fv(N_2)$.
\end{itemize}
\end{proof}

In the following lemma, by $\rightarrow^*$ we denote the reflexive
and transitive closure of the reductions and equivalences of
$\rcl$-calculus, i.e., $\rightarrow^* \triangleq
(\rightarrow_{\rcl} \cup \equiv_{\rcl})^*$.

\begin{lemma}
\label{lemma:add-red}
\rule{0in}{0in}
\begin{itemize}
\item[(i)] $M\isub{\weak{y}{N}}{x} \rightarrow^*
\weak{y}{M\isub{N}{x}}$ \item[(ii)]
$\cont{y}{y_{1}}{y_{2}}{M\isub{N}{x}} \rightarrow^*
M\isub{\cont{y}{y_{1}}{y_{2}}{N}}{x}$, for $y_1,y_2 \notin Fv(M)$.
\end{itemize}
\end{lemma}

\begin{proof}
\rule{0in}{0in}
The proof is by induction on the structure of the
term $M$.
\begin{itemize}
\item[(i)]
\begin{itemize}
\item $M = x$. Then $M\isub{\weak{y}{N}}{x} = x\isub{\weak{y}{N}}{x} \triangleq \weak{y}{N}
    \triangleq \weak{y}{x\isub{N}{x}} = \weak{y}{M\isub{N}{x}}$.
\item $M = \lambda z.P$. Then $M\isub{\weak{y}{N}}{x} = (\lambda z.P)\isub{\weak{y}{N}}{x} \triangleq
    \lambda z.P\isub{\weak{y}{N}}{x} \rightarrow_{IH} \lambda z.(\weak{y}{P\isub{N}{x}})
    \rightarrow_{\omega_{1}} \weak{y}{(\lambda z.P\isub{N}{x})}
    \triangleq \weak{y}{(\lambda z.P)\isub{N}{x}} = \weak{y}{M\isub{N}{x}}$.
\item $M=PQ$. We will treat the case when $x \not\in Fv(Q)$. The case when  $x \not\in Fv(P)$ is analogous. \\
    Then  $M\isub{\weak{y}{N}}{x} =  (PQ)\isub{\weak{y}{N}}{x} \triangleq P\isub{\weak{y}{N}}{x}Q
    \rightarrow_{IH} (\weak{y}{P\isub{N}{x}})Q \\ \rightarrow_{\omega_{2}} \weak{y}{(P\isub{N}{x}Q)}
    \triangleq \weak{y}{(PQ)\isub{N}{x}} = \weak{y}{M\isub{N}{x}}$.
\item $M = \weak{z}{P}$. Then $M\isub{\weak{y}{N}}{x} = (\weak{z}{P})\isub{\weak{y}{N}}{x}
    \triangleq \weak{z}{P}\isub{\weak{y}{N}}{x} \rightarrow_{IH} \weak{z}{\weak{y}{P\isub{N}{x}}}
    \equiv_{\epsilon_{1}} \weak{y}{\weak{z}{P\isub{N}{x}}} = \weak{y}{M\isub{N}{x}}$.
\item $M = \weak{x}{P}$. Then $M\isub{\weak{y}{N}}{x} = (\weak{x}{P})\isub{\weak{y}{N}}{x}
    \triangleq \weak{Fv(\weak{y}{N})}{P} = \weak{y}{\weak{Fv(N)}{P}} \triangleq
    \weak{y}{(\weak{x}{P})\isub{N}{x}} =
    \weak{y}{M\isub{N}{x}}$, since $x\not\in Fv(P)$.
\item $M = \cont{z}{z_{1}}{z_{2}}{P}$. Then $M\isub{\weak{y}{N}}{x}
    = (\cont{z}{z_{1}}{z_{2}}{P})\isub{\weak{y}{N}}{x}
    \triangleq \cont{z}{z_{1}}{z_{2}}{P}\isub{\weak{y}{N}}{x}
    \rightarrow_{IH} \cont{z}{z_{1}}{z_{2}}{(\weak{y}{P}\isub{N}{x})}
    \rightarrow_{\gamma\omega_{1}} \weak{y}{(\cont{z}{z_{1}}{z_{2}}{P\isub{N}{x}})}
    = \weak{y}{M\isub{N}{x}}$.
\item $M = \cont{x}{x_{1}}{x_{2}}{P}$. Then $M\isub{\weak{y}{N}}{x}
    = (\cont{x}{x_{1}}{x_{2}}{P})\isub{\weak{y}{N}}{x}
    \triangleq\\ \cont{Fv(\weak{y}{N})}{Fv(\weak{y_{1}}{N_{1}})}{Fv(\weak{y_{2}}{N_{2}})}
    {P}[\weak{y_{1}}{N_{1}} / {x_{1}}, \weak{y_{2}}{N_{2}} / {x_{2}}]
    \rightarrow_{IH}\\ \cont{Fv(\weak{y}{N})}{Fv(\weak{y_{1}}{N_{1}})}{Fv(\weak{y_{2}}{N_{2}})}
    {\weak{\weak{y_{1}}{y_{2}}}{P[N_{1} / x_{1}, N_{2} / x_{2}]}}
    =\\ \cont{Fv(N)}{Fv(N_{1})}{Fv(N_{2})}{\cont{y}{y_{1}}{y_{2}}
    {\weak{\weak{y_{1}}{y_{2}}}{P[N_{1} / x_{1}, N_{2} / x_{2}]}}}
    \rightarrow_{\gamma\omega_{2}}\\
    \cont{Fv(N)}{Fv(N_{1})}{Fv(N_{2})}
    {\weak{y}{P[N_{1} / x_{1}, N_{2} / x_{2}]}}
    \rightarrow_{\gamma\omega_{1}}\\ \weak{y}{\cont{Fv(N)}{Fv(N_{1})}{Fv(N_{2})}
    {P[N_{1} / x_{1}, N_{2} / x_{2}}]}
    \triangleq \weak{y}{(\cont{x}{x_{1}}{x_{2}}{P\isub{N}{x}})}
    = \weak{y}{M\isub{N}{x}}$.
\end{itemize}

\item[(ii)]
\begin{itemize}
\item $M = x$. Then $\cont{y}{y_1}{y_2}{M\isub{N}{x}} =
    \cont{y}{y_1}{y_2}{x\isub{N}{x}} \triangleq \cont{y}{y_1}{y_2}{N}
    \triangleq x\isub{\cont{y}{y_1}{y_2}{N}}{x} = M\isub{\cont{y}{y_1}{y_2}{N}}{x}$.
\item $M = \lambda z.P$. Then $\cont{y}{y_1}{y_2}{M\isub{N}{x}} =
    \cont{y}{y_1}{y_2}{(\lambda z.P)\isub{N}{x}} \triangleq
    \cont{y}{y_1}{y_2}{\lambda z.P\isub{N}{x}} \\ \rightarrow_{\gamma_1} \lambda z.\cont{y}{y_1}{y_2}{P\isub{N}{x}}
    \rightarrow_{IH} \lambda z.P\isub{\cont{y}{y_1}{y_2}{N}}{x}
    \triangleq ( \lambda z.P)\isub{\cont{y}{y_1}{y_2}{N}}{x} =\\ M\isub{\cont{y}{y_1}{y_2}{N}}{x}$.
\item $M=PQ$, $x \not\in Fv(Q)$. The case when  $x \not\in Fv(P)$ is analogous. \\
    Then  $\cont{y}{y_1}{y_2}{M\isub{N}{x}} =  \cont{y}{y_1}{y_2}{(PQ)\isub{N}{x}} \triangleq \cont{y}{y_1}{y_2}{P\isub{N}{x}Q}
    \rightarrow_{\gamma_2} (\cont{y}{y_1}{y_2}{P\isub{N}{x}})Q \\ \rightarrow_{IH} P\isub{\cont{y}{y_1}{y_2}{N}}{x}Q
    \triangleq (PQ)\isub{\cont{y}{y_1}{y_2}{N}}{x} = M\isub{\cont{y}{y_1}{y_2}{N}}{x}$.
\item $M = \weak{z}{P}$, where $z \neq x,y_1,y_2$. Then
    $\cont{y}{y_1}{y_2}{M\isub{N}{x}} =
    \cont{y}{y_1}{y_2}{(\weak{z}{P})\isub{N}{x}} \triangleq
    \cont{y}{y_1}{y_2}{\weak{z}{P}\isub{N}{x}} \rightarrow_{\gamma\omega_1}
    \weak{z}{\cont{y}{y_1}{y_2}{P\isub{N}{x}}}
    \rightarrow_{IH} \weak{z}{P\isub{\cont{y}{y_1}{y_2}{N}}{x}}
    \triangleq (\weak{z}{P})\isub{\cont{y}{y_1}{y_2}{N}}{x} = M\isub{\cont{y}{y_1}{y_2}{N}}{x}$.
\item $M = \weak{x}{P}$. Then
    $\cont{y}{y_1}{y_2}{M\isub{N}{x}} =
    \cont{y}{y_1}{y_2}{(\weak{x}{P})\isub{N}{x}} \triangleq
    \cont{y}{y_1}{y_2}{\weak{Fv(N)}{P}}.$
    Since $y_1,y_2 \in Fv(N)$ we have that
    $\cont{y}{y_1}{y_2}{\weak{y_1}{\weak{y_2}{\weak{Fv(N)\setminus\{y_1,y_2\}}{P}}}}\rightarrow_{\gamma\omega_2}
    \weak{y}{\weak{Fv(N)\setminus\{y_1,y_2\}}{P}}.$\\
    On the other hand,
    $M\isub{\cont{y}{y_1}{y_2}{N}}{x} =
    (\weak{x}{P})\isub{\cont{y}{y_1}{y_2}{N}}{x} \triangleq
    \weak{Fv(\cont{y}{y_1}{y_2}{N})}{P} =
    \weak{y}{\weak{Fv(N)\setminus\{y_1,y_2\}}{P}},$ so the
    proposition is proved.
\item $M = \cont{z}{z_{1}}{z_{2}}{P}$. Then
    $\cont{y}{y_1}{y_2}{M\isub{N}{x}} =
    \cont{y}{y_1}{y_2}{(\cont{z}{z_{1}}{z_{2}}{P})\isub{N}{x}}
    \triangleq\\ \cont{y}{y_1}{y_2}{\cont{z}{z_{1}}{z_{2}}{P}\isub{N}{x}}
    \equiv_{\rcl} \cont{z}{z_1}{z_2}{\cont{y}{y_{1}}{y_{2}}{P}\isub{N}{x}}
    \rightarrow_{IH}
    \cont{z}{z_{1}}{z_{2}}{P\isub{\cont{y}{y_1}{y_2}{N}}{x}}
    \triangleq
    (\cont{z}{z_1}{z_2}{P})\isub{\cont{y}{y_1}{y_2}{N}}{x}
    = M\isub{\cont{y}{y_1}{y_2}{N}}{x}$.

\item
    $\cont{y}{y_1}{y_2}{M\isub{N}{x}} =
    \cont{y}{y_1}{y_2}{(\cont{x}{x_{1}}{x_{2}}{P})\isub{N}{x}}
    \triangleq \\
    \cont{y}{y_1}{y_2}{\cont{Fv(N)}{Fv(N_{1})}{Fv(N_{2})}{P}[N_{1}/x_{1},N_{2}/x_{2}]} =  \\
    \cont{y}{y_{1}}{y_{2}}{\cont{y_{1}}{y_{1}'}{y_{2}'}{\cont{y_{2}}{y_{1}''}{y_{2}''}{\cont{Fv(N) \setminus \{y_{1},y_{2}\}}{Fv(N_{1}) \setminus \{y_{1}',y_{1}''\}}{Fv(N_{2}) \setminus \{y_{2}',y_{2}''\}}{P[N_{1}/x_{1},N_{2}/x_{2}]}}}}   \equiv_{\rcl}\\
    \cont{y}{y_{1}}{y_{1}'}{\cont{y_{1}}{y_{2}}{y_{2}'}{\cont{y_{2}}{y_{1}''}{y_{2}''}{\cont{Fv(N) \setminus \{y_{1},y_{2}\}}{Fv(N_{1}) \setminus \{y_{1}',y_{1}''\}}{Fv(N_{2}) \setminus \{y_{2}',y_{2}''\}}{P[N_{1}/x_{1},N_{2}/x_{2}]}}}} \equiv_{\rcl}\\
    \cont{y}{y_{1}}{y_{1}'}{\cont{y_{1}}{y_{2}}{y_{1}''}{\cont{y_{2}}{y_{2}'}{y_{2}''}{\cont{Fv(N) \setminus \{y_{1},y_{2}\}}{Fv(N_{1}) \setminus \{y_{1}',y_{1}''\}}{Fv(N_{2}) \setminus \{y_{2}',y_{2}''\}}{P[N_{1}/x_{1},N_{2}/x_{2}]}}}} \equiv_{\rcl}\\
    \cont{y}{y_{1}}{y_{2}}{\cont{y_{1}}{y_{1}'}{y_{1}''}{\cont{y_{2}}{y_{2}'}{y_{2}''}{\cont{Fv(N) \setminus \{y_{1},y_{2}\}}{Fv(N_{1}) \setminus \{y_{1}',y_{1}''\}}{Fv(N_{2}) \setminus \{y_{2}',y_{2}''\}}{P[N_{1}/x_{1},N_{2}/x_{2}]}}}} \equiv_{\rcl}\\
    \cont{y}{y_1}{y_2}{\cont{Fv(N) \setminus \{y_{1},y_{2}\}}{Fv(N_{1}) \setminus \{y_{1}',y_{1}''\}}{Fv(N_{2}) \setminus \{y_{2}',y_{2}''\}}{\cont{y_{2}}{y_{2}'}{y_{2}''}{\cont{y_{1}}{y_{1}'}{y_{1}''}{P[N_{1}/x_{1},N_{2}/x_{2}]}}}} \rightarrow_{IHx2} \\
    \cont{y}{y_1}{y_2}{\cont{Fv(N) \setminus \{y_{1},y_{2}\}}{Fv(N_{1}) \setminus \{y_{1}',y_{1}''\}}{Fv(N_{2}) \setminus \{y_{2}',y_{2}''\}}{P[(\cont{y_{1}}{y_{1}'}{y_{1}''}{N_{1}})/x_{1},(\cont{y_{2}}{y_{2}'}{y_{2}''}{N_{2}})/x_{2}]}}
   $

    On the other hand, rewriting the right hand side yields:\\
    $M\isub{\cont{y}{y_1}{y_2}{N}}{x} =
    (\cont{x}{x_1}{x_2}{P})\isub{\cont{y}{y_1}{y_2}{N}}{x} \triangleq\\
    \cont{Fv(\cont{y}{y_{1}}{y_{2}}{N})}{Fv(\cont{y_{3}}{y_{1}}{y_{2}}{N'_{1}})}
    {Fv(\cont{y_{4}}{y_{1}}{y_{2}}{N'_{2}})}{P[(\cont{y_{3}}{y_{1}}{y_{2}}{N'_{1}})/x_{1},(\cont{y_{4}}{y_{1}}{y_{2}}{N'_{2}})/x_{2}]} \\
    $ By renaming $y_1 \to y'_1$ and $y_2 \to y''_1$ in
    $\cont{y_{3}}{y_{1}}{y_{2}}{N'_{1}}$ and $y_1 \to y'_2$ and $y_2
    \to y''_2$ in $\cont{y_{4}}{y_{1}}{y_{2}}{N'_{2}}$
    we get\\
    $
    \cont{Fv(\cont{y}{y_{1}}{y_{2}}{N})}{Fv(\cont{y_{3}}{y'_{1}}{y''_{1}}{N_{1}})}
    {Fv(\cont{y_{4}}{y'_{2}}{y''_{2}}{N_{2}})}{P[(\cont{y_{3}}{y'_{1}}{y''_{1}}{N_{1}})/x_{1},(\cont{y_{4}}{y'_{2}}{y''_{2}}{N_{2}})/x_{2}]}
    $\\
    where $N'_1 [y'_1/y_1, y''_1/y_2] = N_1$ and $N'_2 [y'_2/y_1, y''_2/y_2] = N_2.$\\

Finally, by renaming $y_3 \to y_1$ and   $y_4 \to y_2$ we get \\
    $
    \cont{Fv(\cont{y}{y_{1}}{y_{2}}{N})}{Fv(\cont{y_{1}}{y'_{1}}{y''_{1}}{N_{1}})}
    {Fv(\cont{y_{2}}{y'_{2}}{y''_{2}}{N_{2}})}{P[(\cont{y_{1}}{y'_{1}}{y''_{1}}{N_{1}})/x_{1},(\cont{y_{2}}{y'_{2}}{y''_{2}}{N_{2}})/x_{2}]} = \\
    \cont{y}{y_1}{y_2}{\cont{Fv(N) \setminus \{y_{1},y_{2}\}}{Fv(N_{1}) \setminus \{y_{1}',y_{1}''\}}{Fv(N_{2}) \setminus \{y_{2}',y_{2}''\}}{P[(\cont{y_{1}}{y_{1}'}{y_{1}''}{N_{1}})/x_{1},(\cont{y_{2}}{y_{2}'}{y_{2}''}{N_{2}})/x_{2}]}}
    $

which completes the proof.
\end{itemize}
\end{itemize}
\end{proof}

Since the last case of the previous lemma is a bit tricky, let us
illustrate it with the following example.

\begin{example}
  \begin{rm}
    Let $M = \cont{x}{x_{1}}{x_{2}}{x_1x_2}$ and $N = y_1y_2$. Then\\
    $\cont{y}{y_{1}}{y_{2}}{M}\isub{N}{x} = \cont{y}{y_1}{y_2}{\cont{x}{x_{1}}{x_{2}}{x_1x_2}\isub{(y_1y_2)}{x}} \triangleq \\
    \cont{y}{y_1}{y_2}{\cont{Fv(y_1y_2)}{Fv(z_1z_2)}{Fv(w_1w_2)}{x_1x_2}[(z_1z_2)/x_{1},(w_1w_2)/x_{2}]} =\\
    \cont{y}{y_1}{y_2}{\cont{y_1}{z_1}{w_1}\cont{y_2}{z_2}{w_2}{x_1x_2}[(z_1z_2)/x_{1},(w_1w_2)/x_{2}]}
    \equiv_{\rcl,(3 \times \epsilon_3)}\\
    \cont{y}{y_1}{y_2}{} \cont{y_1}{z_1}{z_2}{}
    \cont{y_2}{w_1}{w_2}{(z_1z_2)(w_1w_2)} = M_1$.\\
    On the other hand:\\
    $\cont{x}{x_{1}}{x_{2}}{x_1x_2}\isub{\cont{y}{y_1}{y_2}{y_1y_2}}{x} \triangleq\\
    \cont{Fv(\cont{y}{y_{1}}{y_{2}}{y_1y_2})}{Fv(\cont{y_{3}}{y_{1}}{y_{2}}{y_1y_2})}
    {Fv(\cont{y_{4}}{y_{1}}{y_{2}}{y_1y_2})}{(x_1x_2)[(\cont{y_{3}}{y_{1}}{y_{2}}{y_1y_2})/x_{1},(\cont{y_{4}}{y_{1}}{y_{2}}{y_1y_2})/x_{2}]} =\\
    \cont{y}{y_3}{y_4}{(\cont{y_{3}}{y_{1}}{y_{2}}{y_1y_2})(\cont{y_{4}}{y_{1}}{y_{2}}{y_1y_2})}.$\\
    By renaming $y_1 \to z_1$, $y_2 \to z_2$ in the first bracket, and
    $y_1 \to w_1$, $y_2 \to w_2$ in the second one we obtain:
    $\cont{y}{y_3}{y_4}{(\cont{y_{3}}{z_{1}}{z_{2}}{z_1z_2})(\cont{y_{4}}{w_{1}}{w_{2}}{w_1w_2})}.$\\
    By renaming $y_3 \to y_1$, $y_4 \to y_2$ we get
    $\cont{y}{y_1}{y_2}{(\cont{y_{1}}{z_{1}}{z_{2}}{z_1z_2})(\cont{y_{2}}{w_{1}}{w_{2}}{w_1w_2})}  = M_2.$\\
    Finally, $ M_1 \to_{\gamma_2, \gamma_3} M_2.  $
  \end{rm}

\end{example}

\subsection{Intersection types for $\rcl$}
\label{subsec:rcl-inttypes}

In this subsection we introduce an intersection type assignment
system which assigns \emph{strict types} to $\rcl$-terms. Strict
types were proposed in~\cite{bake92} and used
in~\cite{espiivetlika08} for characterisation of strong
normalisation in $\lG$-calculus.

The syntax of types is defined as follows:
$$
\begin{array}{lccl}
\textrm{Strict types}  & \tS & ::= & p \mid \tA \to \tS\\
\textrm{Types}      & \tA & ::= & \cap^n_i \tS_i
\end{array}
$$
\noindent where $p$ ranges over a denumerable set of type atoms,
and $\cap_i^n \tS_i$ stands for $\tS_1 \cap \ldots \cap
\tS_n,\;n \geq 0$.
Particularly, if $n=0$, then $\cap_i^0 \tS_i$ represents the
\emph{neutral element} for the intersection operator, denoted by $\top$.\\
We denote types with $\tA,\tB,\tC...$, strict types with
$\tS,\tT,\tU...$ and the set of all types by $\mathsf{Types}$. We
assume that the intersection operator is idempotent, commutative
and associative.  We also assume that
intersection has priority over the arrow operator. Hence, we will
omit parenthesis in expressions like $(\cap^n_i \tT_i) \to \tS$.

\begin{definition}
\label{def:typass-basis} \rule{0in}{0in}
\begin{itemize}
\item[(i)] A \emph{basic type assignment} is an expression of the form
  $x:\tA$, where $x$ is a term variable and $\tA$ is a type.
\item[(ii)] A \emph{basis} $\Gamma$ is a set $\{x_{1} : \tA_{1},
  \ldots, x_{n} : \tA_{n}\}$ of basic type assignments, where all term
  variables are different.  $Dom(\Gamma)=\{x_{1}, \ldots, x_{n}\}$.  A
  basis extension $\Gamma, x : \tA$ denotes the set $\Gamma \cup
  \{x:\tA\}$, where $x \not \in Dom(\Gamma).$
\item[(iii)] A \emph{bases intersection} is defined as:
   \[
   \Gamma \sqcap \Delta = \{x:\tA \cap \tB\;|\;x:\tA \in \Gamma \; \& \; x:\tB \in \Delta \; \& \; Dom(\Gamma) = Dom(\Delta)\}.
  \]
\item[(iv)]
$
\gtop=\{x:\top\;|\; x \in Dom(\Gamma) \}.
$
\end{itemize}
\end{definition}

In what follows we assume that the bases intersection has priority
over the basis extension, hence the parenthesis in
$\Gamma, (\Delta_1\sqcap \ldots \sqcap \Delta_n)$ will be omitted.
It is easy to show that $\gtop \sqcap \Delta = \Delta$ for
arbitrary bases $\Gamma$ and $\Delta$ that can be intersected,
hence $\gtop$ can be considered the neutral element for the
bases intersection.


\begin{figure}[ht]
\centerline{ \framebox{ $
\begin{array}{c}
\\
\infer[(Ax)]{x:\tS \vdash x:\tS}{}\\ \\
\infer[(\to_I)]{\Gamma \vdash \lambda x.M:\tA \to \tS}
                        {\Gamma,x:\tA \vdash M:\tS} \quad\quad
\infer[(\to_E)]{\Gamma, \dztop \sqcap \Delta_1 \sqcap ...
\sqcap \Delta_n \vdash MN:\tS}
                    {\Gamma \vdash M:\cap^n_i \tT_i \to \tS & \Delta_0 \vdash N:\tT_0\; \ldots\; \Delta_n \vdash N:\tT_n}\\ \\
\infer[(Cont)]{\Gamma, z:\tA \cap \tB \vdash
\cont{z}{x}{y}{M}:\tS}
                    {\Gamma, x:\tA, y:\tB \vdash M:\tS} \quad\quad
\infer[(Weak)]{\Gamma, x:\top \vdash \weak{x}{M}:\tS}
                    {\Gamma \vdash M:\tS}\\
\end{array}
$ }} \caption{$\rcl \cap$: $\rcl$-calculus with intersection
types} \label{fig:typ-rcl-int}
\end{figure}

The type assignment system $\rcl \cap$ is given in Figure~\ref{fig:typ-rcl-int}.
Notice that in the syntax of $\rcl$ there are three
kinds of variables according to the way they are introduced, namely as
a placeholder, as a result of a contraction or as a result of a
weakening. Each kind of a variable receives a specific
type. Variables as placeholders have a strict type, variables resulting from
a contraction have an intersection type and variables resulting from a
weakening have a $\top$ type.  Moreover, notice that intersection
types occur only in two inference rules. In the rule $(Cont)$ the
intersection type is created, this being \emph{the only} place where
this happens. This is justified because it corresponds to the
duplication of a variable. In other words, the control on the
duplication of variables entails the control on the introduction of
intersections in building the type of the term in question. In the
rule $(\to_E)$, intersection appears on the right hand side of~$"|-"$
sign which corresponds to the usage of the intersection type after it
has been created by the rule $(Cont)$ \pierre{or by the rule $(Weak)$
  if $n=0$}.  In this inference rule, the
role of $`D_0$ should be noticed.  It is needed only when $n=0$
to ensure that $N$ has a type, i.e.\ that $N$
is strongly normalizing. Then, in the bottom of the rule, the types
of the free variables of $N$ can be forgotten, hence all the free
variables of $N$ receive the type $\top$.  All the free
variables of the term must occur in the environment (see Lemma~\ref{lem:dom-corr}),
therefore useless
variables occur with the type $\top$.  If $n$ is not~$0$, then $`D_0$
can be any of the other environments and the type of $N$ the
associated type.  Since $`D^{\top}$ is a neutral element for $\sqcap$, then
$`D^{\top}$ disappears in the bottom of the rule.  The case for $n=0$
resembles the rules $(drop)$ and/or \textit{(K-cup)}
in~\cite{lenglescdougdezabake04} and was used to present the two
cases, $n=0$ and $n\neq 0$ in a uniform way.  In the rule $(Weak)$ the
choice of the type of $x$ is~$\top$, since this corresponds to a
variable which does not occur anywhere in~$M$.  The remaining rules,
namely $(Ax)$ and $(\to_I)$ are traditional, i.e.\ they are the same as in
the simply typed $\lambda$-calculus.  Noticed however that the type of
the variable in $(Ax)$ is a strict type.

\begin{lemma}[Domain Correspondence for $\rcl\cap$]
\label{lem:dom-corr}
Let $`G \vdash M:\tS$ be a typing judgment. Then ${x`: Dom(`G)}$
if and only if $x`: Fv(M)$.
\end{lemma}
\begin{proof}
  The rules of Figure~\ref{fig:typ-rcl-int} belong to three categories.
  \begin{enumerate}
  \item \emph{The rules that introduce a variable}. These rules are
    \emph{(Ax)}, $(Cont)$ and $(Weak)$.  One sees that the variable is
    introduced in the environment if and only it is introduced in the
    term as a free variable.
  \item \emph{The rules that remove variables}.  These rules are
    $(\to_I)$ and $(Cont)$.   One sees that the variables are
    removed from the environment if and only if they are removed  from the
    term as a free variable.
  \item \emph{The rule that does not introduce and does not remove a
    variable}.  This rule is $(\to_E)$.
  \end{enumerate}
Notice that $(Cont)$ introduces and removes variables.
\end{proof}
The Generation Lemma makes somewhat more precise the Domain
Correspondence Lemma.

\begin{lemma}[Generation lemma for $\rcl\cap$]
\label{prop:intGL} \rule{0in}{0in}
\begin{enumerate}
%
\item[(i)]\quad $\Gamma \vdash \lambda x.M:\tT\;\;$ iff there
exist $`a$ and $`s$ such that $\;\tT\equiv \tA\rightarrow \tS\;\;$
and $\;\Gamma,x:\tA \vdash M:\tS.$
\item[(ii)]\quad $\Gamma \vdash MN:\tS\;\;$ iff
and there exist $`D_i$ and $\tT_i,\;i  = 0, \ldots, n$ such that
$\Gamma' \vdash M:\cap_{i}^{n}\tT_i\to \tS$ and for all $i \in
\{0, \ldots, n\}$, $\;\Delta_{i} \vdash N:\tT_i$ and $\;\Gamma=\Gamma',
\dztop \sqcap \Delta_{1} \sqcap \ldots \sqcap \Delta_{n}$.
\item[(iii)]\quad $\Gamma \vdash \cont{z}{x}{y}{M}:\tS\;\;$ iff
there exist $`G', `a, `b$ such that
$\;\Gamma=\Gamma',z:`a\cap`b$ \\
and $\;\Gamma', x:\tA, y:\tB \vdash M:\tS.$
\item[(iv)]\quad $\Gamma \vdash \weak{x}{M}:\tS\;\;$ iff
$\;\Gamma=\Gamma', \jel{x:\top}$ and $\;\Gamma' \vdash M:\tS.$
\end{enumerate}
\end{lemma}

\begin{proof} The proof is straightforward since all the rules are syntax directed.
\end{proof}

The proposed system satisfies the following properties.




\begin{lemma}[Substitution lemma for $\rcl\cap$]
\label{prop:sub-lemma} If $\;\Gamma, x:\cap_i^{n} \tT_i \vdash
M:\tS\;$ and for all $i \in \{\jel{0}, \ldots, n\}$, $\;\Delta_i
\vdash N:\tT_i$, then $\;\Gamma,\jel{\dztop \sqcap \Delta_1}
\sqcap ... \sqcap \Delta_n \vdash M\isub{N}{x}:\tS.$
\end{lemma}
\jel{\begin{proof} \rule{0in}{0in} The proof is by induction on
the structure of the term $M$. We only show the interesting cases.
\begin{itemize}
 \item Base case $M = x$. By the axiom $x:\tT \vdash x:\tT$ where
 $\tT = \cap_i^1\tT_i$, i.e.\ $n=1$, hence the second assumption
 is $\;\Delta \vdash N:\tT$ which proves the case since $N
 \triangleq x\isub{N}{x}$.
 \item $M = \weak{x}{P}$. Now we assume $\Gamma, x:\cap_i^0\tT_i
 \vdash \weak{x}{P}:\tS$ and $\Delta_i \vdash N:\tT_i$ for all $i
 \in \{0,\ldots,0\}$, in other words $\Gamma, x:\top
 \vdash \weak{x}{P}:\tS$ and $\Delta_0 \vdash N:\tT_0$ (i.e.\ $N$ is typeable). By
 Generation lemma~\ref{prop:intGL}$(iv)$ we get $\Gamma \vdash
 P:\tS$. Since $Dom(\Delta_0) = Fv(N)$ by applying the $(Weak)$
 rule multiple times we get $\Gamma,\dztop \vdash \weak{Fv(N)}{P} \vdash \tS$
 which is exactly what we want to prove.
 \item $M = \cont{x}{x_1}{x_2}{P}$. From $\Gamma, x:\cap_i^n\tT_i
 \vdash \cont{x}{x_1}{x_2}{P}:\tS$, by  Generation
 lemma~\ref{prop:intGL}$(iii)$ we get that $\cap_i^n\tT_i = \cap_{i=1}^m\tT_i \cap \cap_{i=m+1}^n\tT_i$ for some $m<n$ and $\Gamma,
 x_1:\cap_i^m\tT_i, x_2:\cap_{i=m+1}^n\tT_i \vdash P:\tS$. From
 the other assumption $\Delta_i \vdash N:\tT_i$ for all $i
 \in \{1,\ldots,n\}$, by renaming the variables in $\Delta_i$
 (i.e.\ the free variables of $N$) we get two different sets of
 sequents: $\Delta'_j \vdash N_1:\tT_j$ for $j=
 1,\ldots,m$ and $\Delta''_k \vdash N_2:\tT_k$ for $k=m+1,\ldots,n$.
 By applying IH twice, we get $\Gamma,{\Delta'}_0^{\top} \sqcap \Delta'_1 \sqcap
 \ldots \sqcap \Delta'_m, {\Delta''}_0^{\top} \sqcap \Delta''_{m+1} \sqcap
 \ldots \sqcap \Delta''_n \vdash (P
 \isub{N_1}{x_1})\isub{N_2}{x_2}:\tS$. Now, we apply the definition of the parallel substitution, and perform contraction
 on all pairs of corresponding (i.e.\ obtained by the renaming of the same variable) elements of $\Delta'_j$ and
 $\Delta''_k$ by introducing again the original names of the free variables of $N$ from $\Delta_i$ and finally get what we need:
\[\Gamma,\dztop \sqcap \Delta_1 \sqcap
 \ldots \sqcap \Delta_n \vdash
 \cont{Fv(N)}{Fv(N_1)}{Fv(N_2)}P[N_1/x_1,N_2/x_2]:\tS.
 \]
\end{itemize}
\end{proof}
}

\begin{proposition}[Subject reduction and equivalence]
\label{prop:sr} For every $\rcl$-term $M$: if $\;{\Gamma \vdash
M:\tS}\;$ and $M \to M'$ or $M\equiv M$, then $\;\Gamma \vdash
M':\tS.$
\end{proposition}
\begin{proof}
The proof is done by the case analysis on the applied reduction. Since
 the property is stable by context, we can without losing
 generality assume that the
 reduction takes place at the outermost position of the term.
 Here we just show several cases. We will use GL as an abbreviation for Generation lemma~\ref{prop:intGL}.
  \begin{itemize}
  \item Case $(\beta)$: Let $\Gamma \vdash (\lambda x.M)N:\tS$. We
  want to show that $\Gamma \vdash M\isub{N}{x}:\tS$.
  From $\Gamma \vdash (\lambda x.M)N:\tS\;$ and from GL(ii) it follows that
  $\Gamma = \Gamma',\dztop \sqcap \Delta_1 \sqcap \ldots \sqcap \Delta_n$,  and that there is a type $\cap_i^{n} \tT_i$ such that
  for all $i=0, \ldots, n$, $\Delta_i \vdash N:\tT_i,\;$ and
  $\Gamma' \vdash \lambda x.M:\cap_i^n\tT_i \to \tS$. Further,
  by GL(i) we have that $\Gamma',x: \cap_i^n\tT_i \vdash M:\tS$.
  Now, all the assumptions of Substitution lemma~\ref{prop:sub-lemma}
  hold, yielding $\Gamma',\dztop \sqcap \Delta_1 \sqcap \ldots \sqcap \Delta_n \vdash M\isub{N}{x}:\tS$
  which is exactly what we need, since $\Gamma = \dztop \sqcap \Gamma',\Delta_1 \sqcap \ldots \sqcap \Delta_n$.
  \item Case $(\gamma\omega_2)$: Let $\Gamma \vdash
 \cont{x}{x_1}{x_2}{\weak{x_1}{M}}:\tS$.
 We are showing that $\Gamma \vdash M\isub{x}{x_2}:\tS$.\\
 From the first sequent by GL(iii) we have that $\Gamma =
 \Gamma',x:\tA \cap \tB$ and $\Gamma', x_1:\tA,x_2:\tB
 \vdash\weak{x_1}{M}:\tS$. Further, by GL(iv) we conclude that $\tA \equiv \top$, $x:\top \cap \tB \equiv \tB$ and $\Gamma', x_2:\tB \vdash
 M:\tS$. \jel{Since $\tB= \cap_i^n \tT_i$ for some $n \geq 0$, by applying Substitution lemma~\ref{prop:sub-lemma} to  $\Gamma', x_2:\tB \vdash
 M:\tS$ and $x:\tT_i \vdash x:\tT_i,\;i=0,\ldots,n$} we
 get $\Gamma \vdash M\isub{x}{x_2}:\tS$.
 \item The other rules are easy since they do not essentially
 change the structure of the term.
\end{itemize}
\end{proof}

Due to this property, equivalent terms have the same type.
\subsection{Typeability $\Rightarrow$ SN in $\rcl \cap$}
\label{sec:reducibility}

In various type assignment systems, the \emph{reducibility method} can be used to prove many reduction properties of typeable terms.
It was first introduced by Tait~\cite{tait67} for proving the strong normalisation
of simply typed $\lambda$-calculus, and developed further to prove \emph{strong normalisation} of various calculi in~\cite{tait75,gira71,kriv90,ghil96,ghillika02}, \emph{confluence} (the Church-Rosser property)
of $\beta \eta$-reduction in~\cite{kole85,statI85,mitc90,mitc96,ghillika02}
and to characterise certain classes of $\lambda$-terms
such as strongly normalising, normalising, head normalising,
and weak head normalising terms (and their persistent versions) by their typeability in various intersection type systems in~\cite{gall98,dezahonsmoto00,dezaghil02,dezaghillika04}.

The main idea of the reducibility method is to interpret types by suitable sets of lambda terms which satisfy some realizability properties
and prove the soundness of type assignment with respect to these interpretations.
A consequence of soundness is that every typeable term belongs to the interpretation of its type, hence satisfying a desired reduction property.

In the remainder of the paper we consider $\LR$ as the {\em applicative structure\/} whose domain are $\rcl$-terms and where the application is just the application of $\rcl$-terms.
The set of \emph{strongly normalizing terms} is defined as the smallest subset of $\LR$ such that:

\begin{center}
  \prooftree M' \in \SN \qquad M \rightarrow M' %
  \justifies M \in \SN
  \endprooftree
\end{center}

\begin{definition}
\label{def:fsto}
For $\vM, \vN \subseteq \LR$, we define $\vM \fsto \vN \subseteq
\LR$ as
$$\vM \fsto \vN =  \{N  \in \LR \mid \forall M \in \vM \quad NM \in \vN )\}.$$
\end{definition}

\begin{definition}

\label{def:typeInt} The type interpretation  $\ti{-} : \mathsf{Types} \to
2^{\LR}$ is defined by:
\begin{itemize}
    \item[($I 1$)] $\ti{p} = \SN$, where $p$ is a type atom;
    \item[($I 2$)] $\ti{\tA \to \tS} = \ti{\tA} \fsto \ti{\tS}$;
    \item[($I 3$)] $\ti{\cap^n_i \tS_i}  = \cap^n_i \ti{\tS_i}$  and $\ti{\cap^0_i \tS_i}  = \SN$.
\end{itemize}

\end{definition}


Next, we introduce the notions of \emph{variable property, reduction property,
expansion property,} \emph{weakening property} and \emph{contraction property.}
Variable property and expansion property correspond to the saturation property given in~\cite{bare92}, whereas reduction property corresponds to the property CR~2 in Chapter 6 of~\cite{GirardLafontTaylor89}.
To this aim we
will use the following notation: recall that $\circledR$ denotes the set of reductions
given in Figure~\ref{fig:red-rcl}. If $`m \in \circledR$, then $\rdx_{`m}^{\varsigma}$
denotes a redex, that is a term which is an instance by the meta-substitution $\varsigma$
of the left hand side of the reduction $`m$. Whereas $\ctr_{`m}^{\varsigma}$ denotes the instance
of the right hand side of the same reduction $`m$ by the same meta-substitution
$\varsigma$.\footnote{Meta-substitution is a substitution that assigns values to
  meta-variables.}

\begin{definition}\label{def:var+sat}
\rule{0in}{0in}
\begin{itemize}
\item A set $\vX \subseteq \LR$ satisfies the \emph{variable property}, notation
$\VAR(\vX)$, if  $\vX$ contains all the terms of the form $x M_1 \ldots M_n$ for $M_i`:
\SN$.
\item A set $\vX \subseteq \LR$ satisfies the \emph{reduction property},
  notation $\RED(\vX)$, if $\vX$ is stable by reduction, in other words $M`:\vX$ and $M"->"M'$ imply
  $M'`:\vX$.
\item A set $\vX \subseteq \LR$ satisfies the \emph{expansion property}, notation
  $\SAT_{`m}(\vX)$ where $`m$ is a rule in $\circledR$, if\footnote{Notice that we do not need a
          condition that $N \in \SN$ in $\SAT_{\beta}(\vX)$, as in ordinary $\lambda$-calculus, since we only work with linear
          terms, hence if the contractum $M\isub{N}{x} \in \SN$, then $N \in \SN$.}: %

      \begin{center}
        \prooftree%
        M_1\in \SN \ldots M_n \in \SN \qquad \ctr_{`m}^\varsigma \ M_1 \ldots M_n \in \vX %
        \using \SAT_{`m}(\vX) %
        \justifies \rdx_{`m}^{\varsigma} \ M_1 \ldots M_n \in \vX.
        \endprooftree
      \end{center}
\item A set $\vX \subseteq \LR$ satisfies the \emph{weakening property}, notation $\WEAK(\vX)$ if:
    \begin{center}
      \prooftree M \in \vX %
      \using \WEAK(\vX) %
      \justifies \weak{x}{M} \in \vX.
      \endprooftree
    \end{center}
\item A set $X \subseteq \LR$ satisfies the \emph{contraction property}, notation $\CONT(\vX)$ if:
      \begin{center}
        \prooftree %
        M \in \vX %
        \using \CONT(\vX) %
        \justifies \cont{x}{y}{z}{M} \in \vX.
        \endprooftree
      \end{center}
  \end{itemize}
\end{definition}

\noindent\textbf{Remark.} In the previous definition (Definition~\ref{def:var+sat}) it is not necessary to explicitly write the conditions about free variables since we work with  $\rcl$-terms.

\begin{definition}[$\circledR$-Saturated set] \label{lem:saturSet}
A set $\vX \subseteq \LR$ is called \emph{$\circledR$-saturated},
if
\begin{itemize}
\item $\vX\subseteq \SN$ and
\item $\vX$ satisfies the variable, reduction, expansion, weakening and contraction properties.
\end{itemize}
\end{definition}


\begin{proposition} \label{prop:saturSets}
Let $\vM, \vN \subseteq \LR$.
\begin{itemize}
    \item[(i)] $\SN$ is $\circledR$-saturated.
    \item[(ii)] If $\vM$ and $\vN$ are $\circledR$-saturated, then $\vM \fsto \vN$ is $\circledR$-saturated.
    \item[(iii)] If $\vM$ and $\vN$ are $\circledR$-saturated, then $\vM \cap \vN$ is $\circledR$-saturated.
    \item[(iv)] For all types $\varphi \in \tlam$, $\ti{\varphi}$ is $\circledR$-saturated.
\end{itemize}
\end{proposition}


\begin{proof}
\rule{0in}{0in}
%
(i)
\begin{itemize}
\item
$\SN \subseteq \SN$,  $\VAR(\SN)$ and $\RED(\SN)$ trivially hold.
\item
$\SAT_{\beta}(\SN)$.
Suppose that
	$M\isub{N}{x} M_1 \ldots M_n \in \SN$ and $M_1, \ldots, M_n \in \SN.$
Since $M\isub{N}{x}$ is a subterm of a term in $\SN$, we know that $M \in \SN$. Also,
since $M\isub{N}{x} \in \SN$ and $M$ is linear,  $N \in \SN$. By assumption, $M_{1}, \ldots, M_{n} \in \SN$, so the reductions inside of these terms terminate. After finitely many reduction steps, we obtain
	$(\lambda x.M)N M_1 \ldots M_n \rightarrow \ldots \rightarrow (\lambda x.M')N' M'_1 \ldots M'_n$
where $M \rightarrow M',\; N \rightarrow N', \; M_{1} \rightarrow M'_{1}, \ldots, M_{n} \rightarrow M'_{n}.$ After contracting\\ $(\lambda x.M')N' M'_1 \ldots M'_n$ to $M'\isub{N'}{x} M'_1 \ldots M'_n$, we obtain a reduct of\\ $M\isub{N}{x} M_1 \ldots M_n \in \SN$. Hence, $(\lambda x.M)N M_1 \ldots M_n \in \SN. $
\item
$\SAT_{`m}(\SN)$. Analogous to $\SAT_{\beta}(\SN)$.
\item
$\WEAK(\SN)$.
Suppose that
$M \in \SN$ and $x \not \in Fv(M)$. Then trivially $\weak{x}{M} \in \SN$, since no new redexes are formed.
\item
$\CONT(\SN)$.
Suppose that $M \in \SN,\; y \not = z,\; y, z \in Fv(M),\; x \not \in Fv(M) \setminus \{y,z\}$.
We prove $\cont{x}{y}{z}{M} \in \SN$ by induction on the structure of $M$.
\begin{itemize}
\item $M = \lambda w.N$. Then $N \in \SN$ and $\cont{x}{y}{z}{M} = \cont{x}{y}{z}{(\lambda w.N)} \rightarrow_{\gamma_{1}} \lambda w.\cont{x}{y}{z}{N} \in \SN$, since $\cont{x}{y}{z}{N} \in \SN$ by IH.
\item $M = PQ$. Then $P, Q \in \SN$ and if $y,z \not \in Fv(Q)$, $\cont{x}{y}{z}{M} = \cont{x}{y}{z}{(PQ)} \rightarrow_{\gamma_{2}} (\cont{x}{y}{z}{P})Q \in \SN$, since by IH $\cont{x}{y}{z}{P} \in \SN$.\\ The case of $\rightarrow_{\gamma_{3}}$ reduction is analogous.
\item $M = \weak{w}{N}$. Then $\cont{x}{y}{z}{M} = \cont{x}{y}{z}{(\weak{w}{N})} \rightarrow_{\gamma\omega_{1}} \weak{w}{(\cont{x}{y}{z}{N})}$. By IH  $\cont{x}{y}{z}{N} \in \SN$ and $\weak{w}(\cont{x}{y}{z}{N})$ does not introduce any new redexes.
\item $M = \weak{y}{N}$. Then $\cont{x}{y}{z}{M} = \cont{x}{y}{z}{(\weak{y}{N})} \rightarrow_{\gamma\omega_{2}} N\isub{x}{z} \in \SN$, since $N \in \SN$ by IH.
\end{itemize}
\end{itemize}

(ii)
  \begin{itemize}
  \item $\vM \fsto \vN \subseteq \SN$.  Suppose that $M \in \vM \fsto
    \vN$. 
    Then, for all $N \in \vM,\; MN \in \vN$. Since $\vM$ is $\circledR$-saturated,
    $\VAR(\vM)$ holds so $x \in \vM$ and $Mx \in \vN \subseteq \SN.$ From here we can
    deduce that $M \in \SN$.
  \item $\VAR(\vM \fsto \vN)$.  Suppose that $x \in {\tt var}$, and $M_1, \ldots, M_n \in \SN,
    n \geq 0$, such that $x \cap Fv(M_1) \cap \ldots \cap Fv(M_n) = \emptyset$. We need to show
    that $x M_1 \ldots M_n \in \vM \fsto \vN,$ i.e.\ $\forall N \in \vM$, $x M_1 \ldots
    M_nN \in \vN$. This holds since by IH $\vM \subseteq \SN$ and $\vN$ is
    $\circledR$-saturated, i.e.\ $\VAR(\vN)$ holds.
  \item $\RED(\vM \fsto \vN)$. Let $M `: \vM \fsto \vN$ and $M'$ be such that $M "->" M'$ and
    let $N`:\vM$. We know that $M N`:\vN$ and $M N "->" M' N$.  By IH, $M' N `: \vN$ hence
    $M'`: \vM \fsto \vN$.
  \item $\SAT_{\beta}(\vM \fsto \vN)$.  Suppose that $M\isub{N}{x} M_1 \ldots M_n \in \vM
    \fsto \vN$ and $M_1, \ldots,$ $M_n \in \SN.$ This means that for all $P \in \vM$,
    $M\isub{N}{x} M_1 \ldots M_nP \in \vN.$ But $\vN$ is $\circledR$-saturated, so
    $\SAT_{\beta}(\vN)$ holds and we have that for all $P \in \vN$, $(\lambda x.M)N M_1
    \ldots M_nP \in \vN.$ This means that $(\lambda x.M)N M_1 \ldots M_n \in \vM \fsto
    \vN. $
  \item $\SAT_{`m}(\vM \fsto \vN)$. Analogous to $\SAT_{\beta}(\vM \fsto \vN)$.
  \item $\WEAK(\vM \fsto \vN)$.  Suppose that $M \in \vM \fsto \vN$ and $x \not \in
    Fv(M)$. This means that for all $N \in \vM, MN \in \vN$. But $\vN$ is
    $\circledR$-saturated, i.e.\ $\WEAK(\vN)$ holds, hence $\weak{x}{(MN)} \in \vN$. Also
    $\SAT_{\omega_{2}}(\vN)$ holds so we obtain for all $N \in \vM, (\weak{x}{M})N \in
    \vN$, i.e\ $\weak{x}{M} \in \vM \fsto \vN$.
  \item $\CONT(\vM \fsto \vN)$.  Let $M \in \vM \fsto \vN$. We want to prove that
      ${\cont{x}{y}{z} {M}} \in \vM \fsto \vN$ for $y \not = z,\; y, z \in
      Fv(M)$ and ${x \not \in Fv(M)}$. Let $P$ be any term in $\vM$. We have to prove that
      $({\cont{x}{y}{z} {M}}) \, P \in \vN$.  Since $M \in \vM \fsto \vN$, we know that
      $M\,P\in\vN$. By IH ${\cont{x}{y}{z}{(M\,P)} \in \vN}$. By reduction
      $\gamma_{2}$ and hence by $\RED(\vN)$ we have $({\cont{x}{y}{z} {M}}) \, P \in \vN$.
    Therefore ${\cont{x}{y}{z} {M}} \in \vM \fsto \vN$.
  \end{itemize}

(iii)
\begin{itemize}
\item
$\vM \cap \vN \subseteq \SN$ is straightforward, since $M,N \subseteq \SN$ by IH.
\item
$\VAR(\vM \cap \vN)$. Since $\VAR(\vM)$ and $\VAR(\vN)$ hold, we have that
$\forall M_{1}, \ldots,$ $M_{n}$ $\in \SN$, $n \geq 0$:
$xM_1 \ldots M_{n} \in \vM$ and $xM_1 \ldots M_{n} \in \vN$. We deduce that
$\forall M_{1}, \ldots, M_{n} \in \SN$, $n \geq 0$:
$xM_1 \ldots M_{n} \in \vM \cap \vN$, i.e.\ $\VAR(\vM \cap \vN)$ holds.
\item $\RED(\vM \cap \vN)$ is straightforward.
\item $\SAT_{\beta}(\vM \cap \vN)$ and $\SAT_{`m}(\vM \cap \vN)$ are straightforward.
\item
$\WEAK(\vM \cap \vN)$.
Let $M \in \vM \cap \vN$ and $x \not \in Fv(M)$. Then $M \in \vM$ and $M \in \vN$. Since both $\vM$ and $\vN$ are $\circledR$-saturated $\WEAK(\vM)$ and $\WEAK(\vN)$ hold, hence by IH $\weak{x}{M} \in \vM$ and $\weak{x}{M} \in \vN$, i.e.\ $\weak{x}{M} \in \vM \cap \vN$.
\item
$\CONT(\vM \cap \vN)$.
Suppose that $M \in \vM \cap \vN,\; y \not = z,\; y, z \in Fv(M),\; x \not \in Fv(M) \setminus \{y,z\}$. Since both $\vM$ and $\vN$ are $\circledR$-saturated $\CONT(\vM)$ and $\CONT(\vN)$ hold, hence by IH $\cont{x}{y}{z}{M} \in \vM$ and $\cont{x}{y}{z}{M} \in \vN$, i.e.\ $\cont{x}{y}{z}{M} \in \vM \cap \vN$.
\end{itemize}

(iv)
By induction on the construction of $\varphi \in \mathsf{Types}$.
\begin{itemize}
    \item If $\varphi \equiv p$, $p$ a type atom, then $\ti{\varphi} = \SN$, so it is $\circledR$-saturated using (i).
    \item If $\varphi \equiv \tA \to \tS$, then $\ti{\varphi} = \ti{\tA} \fsto \ti{\tS}$. Since $\ti{\tA}$ and $\ti{\tS}$ are $\circledR$-saturated by IH,  we can use (ii).
    \item If $\varphi \equiv \cap_{i}^{n} \tS_{i}$, then $\ti{\varphi} =\ti{\cap_{i}^{n} \tS_{i}}  = \cap_{i}^{n} \ti{\tS_{i}}$ and for all $i=1, \ldots, n, \ti{\tS_{i}}$ are $\circledR$-saturated by IH, so we can use (iii). If $\varphi \equiv \cap_{i}^{0} \tS_{i}$, then $\ti{\varphi} = \SN$
\end{itemize}

\end{proof}


We further define a {\em valuation of terms\/} $\tei{-}_{\rho}: \LR \to \LR$ and the {\em semantic satisfiability relation\/}
$\models$ connecting the type interpretation with the term valuation.

\begin{definition}  \label{def:val}
Let $\rho : {\tt var} \to \LR$ be a valuation of term variables in
$\LR$. For ${M \in \LR}$, with $Fv(M) = \{x_1, \ldots, x_n\} $ the
\emph{term valuation} $\tei{-}_\rho : \LR \to \LR$ is defined as follows:
$$\tei{M}_{\rho} = M[\isubs{\rho(x_{1})}{x_{1}}, \ldots, \isubs{\rho(x_{n})}{x_{n}}].$$
\noindent providing that $x\not = y \; \Rightarrow \; Fv(\rho(x)) \cap Fv(\rho(y)) = \emptyset$.
\end{definition}

\noindent \textit{Notation:} $`r(N/x)$ is the valuation defined as:
$\displaystyle{
  `r(N/x)(y) = \left\{
    \begin{array}{l}
      `r(y) \qquad \textrm{if~} x \neq y\\
      N \qquad \mbox{otherwise}
     \end{array}
   \right.}$



\begin{lemma} \label{lemma:val}
\rule{0in}{0in}
 \begin{itemize}
    \item[(i)]
    $\tei{MN}_{\rho} = \tei{M}_{\rho}\tei{N}_{\rho}$
    \item[(ii)]
$\tei{\lambda x. M}_{\rho} N   \to \tei{ M}_{\rho(\isubs{N}{x})} $.
    \item[(iii)]
    $\tei{\weak{x}{M}}_{\rho} = \weak{Fv(\rho(x))}{\tei{M}}_{\rho}.$
    \item[(iv)]
    $\tei{\cont{z}{x}{y}{M}}_{\rho} =
    \cont{Fv(N)}{Fv(N_{1})}{Fv(N_{2})}{\tei{M}_{\rho(\isubs{N_{1}}{x},\isubs{N_{2}}{y})}}$\\
    where $N=\rho (z)$, $N_{1}$, $N_{2}$ are obtained from $N$ by renaming its free variables.
\end{itemize}
\end{lemma}

\begin{proof}
\rule{0in}{0in}
 \begin{itemize}
    \item[(i)] Straightforward from the definition of substitution given in Figure~\ref{fig:sub-rcl}.
    \item[(ii)]
	If  $Fv(\lambda x.M) = \{x_{1}, \ldots, x_{n}\}$, then\\
	$\tei{\lambda x.M}_{\rho} N =
	  (\lambda x.M)[\isubs{\rho(x_{1})}{x_{1}}, \ldots, \isubs{\rho(x_{n})}{x_{n}}] N   \to \\
        (M[\isubs{\rho(x_{1})}{x_{1}}, \ldots, \isubs{\rho(x_{n})}{x_{n}}])[\isubs{N}{x}]
    =
	M[ \isubs{\rho(x_{1})}{x_{1}}, \ldots, \isubs{\rho(x_{n})}{x_{n}}, \isubs{N}{x}] = \\
	 \tei{ M}_{\rho(\isubs{N}{x})}$,
    \item[(iii)]
	If  $Fv(M) = \{x_{1}, \ldots, x_{n}\}$, then\\
	$\tei{\weak{x}{M}}_{\rho} =
	  (\weak{x}{M})[\isubs{\rho(x)}{x}, \isubs{\rho(x_{1})}{x_{1}}, \ldots, \isubs{\rho(x_{n})}{x_{n}}] =\\
	  \weak{Fv(\rho(x))}{M[\isubs{\rho(x_{1})}{x_{1}}, \ldots, \isubs{\rho(x_{n})}{x_{n}}]} =
	  \weak{Fv(\rho(x))}{\tei{M}_{\rho}}.$
    \item[(iv)]
	If  $Fv(M) = \{x_{1}, \ldots, x_{n}\}$, then\\
	$\tei{\cont{z}{x}{y}{M}}_{\rho} =
	(\cont{z}{x}{y}{M})[\isubs{N}{z}, \isubs{\rho(x_{1})}{x_{1}}, \ldots, \isubs{\rho(x_{n})}{x_{n}}] =\\
	\cont{Fv(N)}{Fv(N_{1})}{Fv(N_{2})}{M}[\isubs{N_{1}}{x}, \isubs{N_{2}}{y}, \isubs{\rho(x_{1})}{x_{1}}, \ldots, 			 \isubs{\rho(x_{n})}{x_{n}}] =\\
	= \cont{Fv(N)}{Fv(N_{1})}{Fv(N_{2})}{\tei{M}_{\rho(\isubs{N_1}{x},\isubs{N_2}{y})}}.$
\end{itemize}
\end{proof}

\begin{definition} \label{def:model}
\rule{0in}{0in}
\begin{itemize}
    \item [(i)] $\rho \models M : \tS \quad \iff\ \quad \tei{M}_\rho \in \ti{\tS}$;
    \item [(ii)] $\rho \models \Gamma \quad \iff\ \quad (\forall (x:\tA) \in \Gamma) \quad \rho(x)\in \ti{\tA}$;
    \item [(iii)] $\Gamma \models M : \tS \quad \iff\ \quad (\forall \rho, \rho \models \Gamma \Rightarrow
    		\rho \models M : \tS)$.
  \end{itemize}
\end{definition}

\begin{lemma}\label{lem:valcup}
  Let $`G "|=" M:\tS$ and $`D "|="
  M:\tT$, then $$`r "|=" `G \sqcap `D \mbox{ if and only if } `r "|=" `G \mbox{ and } `r "|=" `D.$$

\end{lemma}

\begin{proof}
  The proof is a straightforward consequence of the Definition~\ref{def:typass-basis} of bases intersection $\sqcap$.
\end{proof}

\begin{proposition} [Soundness of $\rcl \cap$] \label{prop:sound}
If $\Gamma \vdash M:\tS$, then $\Gamma \models M:\tS$.
\end{proposition}

\begin{proof}
By induction on the derivation of $\Gamma \vdash M:\tS$.

\begin{itemize}

\item
If the last rule applied is $(Ax)$, i.e.\
	$x:\tS \vdash x:\tS$
the proof is trivial.
%

\item
The last rule applied is $(\to_{I})$, i.e.,\
    $$\Gamma, x:\tA \vdash M:\tS \; \Rightarrow \;
      \Gamma\vdash \lambda x.M:\tA \to \tS.$$
By the IH $\Gamma, x:\tA \models M: \tS$.
Suppose that $\rho \models \Gamma$ and we want to show that $\rho \models \lambda x.M: \tA \to \tS$.
We have to show that
$$\tei{\lambda x.M}_{\rho} \in \ti{\tA \to \tS} = \ti{\tA} \fsto \ti{\tS}\;\; \mbox{ i.e.}\;\;
\forall N \in \ti{\tA}. \; \tei{\lambda x.M}_{\rho}N \in \ti{\tS}.$$
Suppose that $N \in \ti{\tA}$.
We have that $\rho(N/x) \models \Gamma, x:\alpha$ since $\rho \models \Gamma$, $x \not\in \Gamma$ and $\rho(N/x)(x)=N \in \ti{\tA}$. By IH $\rho(N/x) \models M:\sigma$, hence we can conclude that $\tei{M}_{\rho(N/x)} \in \ti{\tS}$. Using Lemma~\ref{lemma:val}(ii) we get $\tei{\lambda x.M}_{\rho}N \to$ $\tei{M}_{\rho(N/x)}$. Since $\tei{M}_{\rho(N/x)} \in \ti{\tS}$ and $\ti{\tS}$ is $\circledR$-saturated,
we obtain $\tei{\lambda x.M}_{\rho}N \in \ti{\tS}$.

\item




The last rule applied is $(\to_{E})$, i.e.\
$$\Gamma \vdash M:\cap_{i}^{n} \tT_i \to \tS,\; \Delta_0 \vdash N:\tT_0\; \ldots \; \Delta_n \vdash N:\tT_n \;			\Rightarrow\; \Gamma, \dztop \sqcap \Delta_1 \sqcap \ldots \sqcap \Delta_{n} \vdash MN:\tS.$$
Let $`r$ be any valuation.  \\
Assuming that
$`G "|-" M: \cap_{i}^{n} \tT_i \to `s, `D_1 "|-"  N:\tT_1,\ldots, `D_n"|-" N:\tT_n$, we have to prove that if
$`r "|=" `G, \dztop \sqcap `D_1  \sqcap ... \sqcap `D_n$, then $`r "|=" M\,N:`s$.

By IH, $\Gamma \models M :\cap_i^n \tT_i \to \tS$ and $\Delta_0 \models N:\tT_0, \ldots, \Delta_n \models N:\tT_n$.
Assume that $\rho \models \Gamma, \dztop \sqcap \Delta_1 \sqcap \ldots \sqcap \Delta_{n}$. This means that $\rho \models \Gamma$ and $\rho \models \dztop \sqcap \Delta_1 \sqcap \ldots \sqcap \Delta_{n}.$
From $\rho \models \Gamma$ we deduce by Definition~\ref{def:model}~(iii) $\rho \models M: \cap_{i}^{n} \tT_i \to \tS$ and by Definition~\ref{def:model}~(i) $\tei{M}_{\rho} \in \ti{\cap_{i}^{n}
  \tT_{i} \to \tS}$.  By Definition~\ref{def:val} $\tei{M}_{\rho} \in \bigcap_{i}^n
\tei{\tT} \fsto \tei{\tS}$.\\
Using Lemma~\ref{lem:valcup} $\rho \models \dztop \sqcap \Delta_1 \sqcap ... \sqcap \Delta_n$ implies
$(\rho \models \dztop) \wedge (\bigwedge_{i=1}^n\rho \models \Delta_{i})$, hence by
Definition~\ref{def:model}~(i) and (iii) we get $(\tei{N}_{\rho} \in \ti{\top}) \wedge
{\bigwedge_{i=1}^n(\tei{N}_{\rho} \in \ti{\tT_{i}})}$, i.e.\ $\tei{N}_{\rho} \in \SN \ \cap\ \cap_{i}^{n }\ti{\tT_{i}} = \cap_{i}^{n }\ti{\tT_{i}}$, since $\ti{\tau_i} \subseteq \SN$ by Proposition~\ref{prop:saturSets}(iv).
By Definition~\ref{def:fsto} of $\fsto\!\!\!\!$, $\tei{M\,N}_{\rho} = \tei{M}_{\rho}\,\tei{N}_{\rho}
\in \tei{\tS}$  and by Definition~\ref{def:model}~(i) ${`r "|=" M\,N:`s}$.

\item
The last rule applied is $(Weak)$, i.e.,\
	$$\Gamma \vdash M:\tS \; \Rightarrow \; \Gamma, x:\top \vdash \weak{x}{M}:\tS.$$
By the IH $\Gamma \models M:\tS$.
Suppose that $\rho \models \Gamma, x:\top$ $\Leftrightarrow$  $\rho \models \Gamma$ and $\rho \models x:\top$. From $\rho \models \Gamma$ we obtain $\tei{M}_{\rho} \in \ti{\tS}$. Using multiple times the weakening property $\WEAK$ and Lemma~\ref{lemma:val}(iii) we obtain $\weak{Fv(\rho(x))}{\tei{M}_{\rho}} = \tei{\weak{x}{M}}_{\rho} \in \ti{\tS}$, since $Fv(\rho(x)) \cap Fv(\tei{M}_{\rho}) = \emptyset$.

\item
The last rule applied is $(Cont)$, i.e.,\
	$$\Gamma, x:\tA, y:\tB \vdash M:\tS \; \Rightarrow \;\Gamma, z:\tA \cap \tB \vdash \cont{z}{x}{y}{M}:\tS.$$
By the IH $\Gamma, x:\tA, y:\tB \models M:\tS$.
Suppose that $\rho \models \Gamma, z:\tA \cap \tB$.
This means that $\rho \models \Gamma$ and $\rho \models z:\tA \cap \tB$ $\Leftrightarrow$ $\rho(z) \in \ti{\tA} \mbox{ and } \rho(z) \in \ti{\tB}$.
For the sake of simplicity let $\rho(z) \equiv N$. We define a new valuation $\rho'$ such that $\rho' = \rho (\isubs{N_{1}}{x}, \isubs{N_{2}}{y})$, where $N_{1}$ and $N_{2}$ are obtained by renaming the free variables of $N$.
Then $\rho' \models \Gamma, x:\tA, y:\tB$ since $x,y \not \in Dom(\Gamma)$, $N_{1} \in \ti{\tA}$ and $N_{2} \in \ti{\tB}$.
By the IH $\tei{M}_{\rho'} = \tei{M}_{\rho(\isubs{N_{1}}{x}, \isubs{N_{2}}{y})}  \in \ti{\tS}$.
Using the contraction property $\CONT$ we have that
$\cont{Fv(N)}{Fv(N_{1})}{Fv(N_{2})}{\tei{M}_{\rho (\isubs{N_{1}}{x}, \isubs{N_{2}}{y})}} =
\tei{\cont{z}{x}{y}{M}}_{\rho} \in \ti{\tS}$.

\end{itemize}
\end{proof}

\begin{theorem} [$\SN$ for $\rcl \cap$] \label{th:typ=>SN}
If $\Gamma \vdash M:\tS$, then $M$ is strongly normalizing, i.e. $M \in \SN$.
\end{theorem}

\begin{proof}
Suppose $\Gamma \vdash M:\tS$. By Proposition~\ref{prop:sound}\; $\Gamma \models M:\tA$. According to Definition~\ref{def:model}(iii), this means that $ (\forall \rho \models \Gamma) \quad \rho \models M : \tS$. We can choose a particular $\rho_{0}(x) = x$ for all $x \in {\tt var}$. By Proposition~\ref{prop:saturSets}(iv), $\ti{\tB}$ is $\circledR$-saturated for each type $\tB$, hence $x = \tei{x}_{\rho_0} \in \ti{\tB}$ (variable condition for $n=0$). Therefore, $\rho_{0} \models \Gamma$ and we can conclude that $\tei{M}_{\rho_{0}} \in \ti{\tS}$. On the other hand, $M = \tei{M}_{\rho_{0}}$ and $\ti{\tS} \subseteq \SN$ (Proposition~\ref{prop:saturSets}), hence $M \in \SN$.
\end{proof}

\subsection{SN $\Rightarrow$ Typeability in $\rcl \cap$}
\label{sec:SNtype}

We want to prove that if a $\rcl$-term is SN, then it is typeable
in the system $\rcl \cap$. We proceed in two steps: 1) we show
that all $\rcl$-normal forms are typeable and 2) we prove the head
subject expansion. First, let us observe the structure of the
$\rcl$-normal forms, given by the following abstract syntax:
$$
\begin{array}{rcl}
 M_{nf} & ::= & x\,|\,\lambda
x.M_{nf}\,|\,\lambda x. \weak{x}{M_{nf}}\,|\, xM_{nf}^{1} \ldots
M_{nf}^{n}\, |\, \\
& & \cont{x}{x_1}{x_2}M_{nf}N_{nf},\,\mbox{with}\;\;x_1
\in Fv(M_{nf}),x_2 \in Fv(N_{nf})\\
W_{nf} & ::= & \weak{x}{M_{nf}}\,|\,\weak{x}{W_{nf}}\\
\end{array}
$$

Notice that it is necessary to distinguish normal forms
$W_{nf}$ since the term $\lambda x. \weak{y}{M_{nf}}$ is not a
normal form, since $\lambda x. \weak{y}{M_{nf}} \to_{\omega_1}
\weak{y}{\lambda x.M_{nf}}$.

\begin{proposition}\label{prop:nf-are-typ}
$\rcl$-normal forms are typeable in the system $\rcl \cap$.
\end{proposition}
\begin{proof}
By induction on the structure of $M_{nf}$ and $W_{nf}$.
\end{proof}

\begin{lemma}[Inverse substitution lemma]\label{prop:inv-subst-lemma}
Let $\;\Gamma \vdash M\isub{N}{x}:\tS\;$ and $N$ typeable. Then,
there are
$\Delta_j$ and $\tT_j, \;j=0,\ldots, n$ such
that $\Delta_j \vdash N:\tT_j,\;$ and $\Gamma', x:\cap_{i}^{n}\tT_i \vdash
M:\tS$, where $\Gamma = \Gamma',\; \dztop \sqcap \Delta_1 \sqcap \ldots \sqcap \Delta_n$.
\end{lemma}

\begin{proof}
By induction on the structure of $M$.
\end{proof}

\begin{proposition}[Head subject expansion]
\label{prop:sub-exp} For every $\rcl$-term $M$: if $M \to M'$, $M$
is a contracted redex and $\;\Gamma \vdash M':\tS\;$, then $\;\Gamma
\vdash M:\tS$, provided that if $M \equiv (\lambda x.N)P
\to_{\beta} N\isub{P}{x} \equiv M'$, $\;P$ is typeable.
\end{proposition}
\begin{proof} By case study according to the applied
reduction.
\end{proof}

\begin{theorem}[SN $\Rightarrow$ typeability]\label{thm:SNtypable}
All strongly normalising $\rcl$-terms are typeable in the
$\rcl\cap$ system.
\end{theorem}
\begin{proof}
The proof is by induction on the length of the longest reduction
path out of a strongly normalising term $M$, with a subinduction
on the size of $M$.
\begin{itemize}
\item If $M$ is a normal form, then $M$ is typeable by
  Proposition~\ref{prop:nf-are-typ}.
\item If $M$ is itself a redex, let $M'$ be the
  term obtained by contracting the redex $M$.  $M'$ is also strongly normalising, hence by
  IH it is typeable.  Then $M$ is typeable, by
  Proposition~\ref{prop:sub-exp}. Notice that, if $M \equiv (\lambda x.N)P
  \to_{\beta} N\isub{P}{x} \equiv M'$, then, by IH, $P$ is typeable, since the length of
  the longest reduction path out of $P$ is smaller than that of $M$, and the size of
  $P$ is smaller than the size of $M$.
\item Next, suppose that $M$ is not itself a redex
  nor a normal form. Then $M$ is of one of the following forms: $\lambda x.N$, $\lambda
  x.\weak{x}{N}$, $xM_{1} \ldots M_{n}$, $\weak{x}{N}$, or
  $\cont{x}{x_1}{x_2}{NP},\;x_1 \in Fv(N),\;x_2 \in Fv(P)$ (where $M_{1}, \ldots,
    M_{n}, \; N$, and $NP$ are \emph{not} normal forms).  $M_{1}, \ldots, M_{n}$ and $NP$ are
    typeable by IH, as subterms of $M$. Then, it is easy to build the typing for
  $M$. For instance, let us consider the case $\cont{x}{x_1}{x_2}{NP}$
    with $x_1 \in Fv(N),\;x_2 \in Fv(P)$.  By induction $NP$ is typeable, hence $N$ is
    typeable with say $\Gamma, x_1:\tB \vdash N:\cap_i^n \tT_i \to \tS$ and $P$ is typeable with say
    $\Delta_j, x_2:\tC_j \vdash P: \tT_j$, for all $j = 0,\ldots,n$. Then using the rule ($E\to$)
    we obtain $\Gamma,\dztop \sqcap \Delta_1 \sqcap \ldots \sqcap \Delta_n, x_1:\tB,
    x_2:\cap_i^n \tC_i \vdash NP: \tS$. Finally, the rule $(Cont)$ yields $`G,
    \dztop \sqcap \Delta_1 \sqcap \ldots \sqcap \Delta_n, x :\tB \cap
    (\cap_i^n \tC_i) \vdash \cont{x}{x_1}{x_2}{NP} :\tS$.
\end{itemize}
\end{proof}

Finally, we can give a characterisation of strong normalisation in $\rcl$-calculus.

\begin{theorem}
In $\rcl$-calculus, the term $M$ is strongly normalising if and only if it is typeable in $\rcl\cap$.
\end{theorem}

\begin{proof}
Immediate consequence of Theorems~\ref{th:typ=>SN} and~\ref{thm:SNtypable}.
\end{proof}

\section{Intersection types for the sequent resource control lambda calculus $\llG$}
\label{sec:lambda_gtz}

In this section we focus on the sequent resource control lambda
calculus $\llG$. First we revisit its syntax and operational
semantics; further we introduce an intersection type assignment
system and finally we prove that typeability in the proposed system
completely characterises the set of strongly normalising
$\llG$-expressions.

\subsection{Resource control sequent lambda calculus $\llG$}
\label{subsec:llG}

The \emph{resource control lambda Gentzen} calculus $\llG$ is
derived from the $\lG$-calculus (more precisely its confluent
sub-calculus $\lG_V$) by adding the explicit operators for
weakening and contraction. It is proposed in
\cite{ghilivetlesczuni11}. The abstract syntax of $\llG$
pre-expressions is the following:
$$
\begin{array}{lcrcl}
\textrm{Pre-values}    &  & F & ::= & x\,|\,\lambda
x.f\,|\,\weak{x}{f}
\,|\, \cont{x}{x_1}{x_2}{f}\\
\textrm{Pre-terms}    &  & f & ::= & F \,|\,fc\\
\textrm{Pre-contexts} & & c & ::= & \bindx
f\,|\,f::c\,|\,\weak{x}{c} \,|\, \cont{x}{x_1}{x_2}{c}
\end{array}
$$
where $x$ ranges over a denumerable set of term variables.

A \emph{pre-value} can be a variable, an abstraction, a weakening
or a contraction; a \emph{pre-term} is either a value or a cut (an
application). A \emph{pre-context} is one of the following:
a~selection, a context constructor (usually called cons), a
weakening on pre-context or a contraction on a pre-context.
Pre-terms and pre-contexts are together referred to as the
\emph{pre-expressions} and will be ranged over by $E$.
Pre-contexts $\weak{x}{c}$ and $\cont{x}{x_1}{x_2}{c}$ behave
exactly like corresponding pre-terms $\weak{x}{f}$ and
$\cont{x}{x_1}{x_2}{f}$ in the untyped calculus, so they will
mostly not be treated separately. The set of free variables of a
pre-expression is defined analogously to the free variables in
$\rcl$-calculus with the following additions:
$$\begin{array}{c} Fv(fc) = Fv(f) \cup Fv(c); \quad Fv(\bindx f) =
Fv(f)\setminus \{x\};\quad Fv(f::c) = Fv(f) \cup Fv(c).
\end{array}$$

Like in $\rcl$-calculus, the set of $\llG$-expressions
(namely values, terms and contexts), denoted by $\LlG \cup \LlGC$,
is a subset of the set of pre-expressions, defined in
Figure~\ref{fig:wf-gtz}. Values are denoted by $T,$ terms by
$t,u,v...$, contexts by $k,k',...$ and expressions by $e,e'$.

\begin{figure}[htpb]
\centerline{ \framebox{ $
    \begin{array}{c}
      \begin{array}{c@{\qquad}c}
        \infer{x \in \LlG}{} &
        \infer{\lambda x.f \in \LlG}
        {f \in \LlG \;\; x \in Fv(f)}
      \end{array}  \\\\
      \begin{array}{c}
        \infer{f c \in \LlG}
        {f \in \LlG\;\; c \in \LlGC \;\; Fv(f) \cap Fv(c) = \emptyset}
      \end{array}
      \\\\
    \begin{array}{c@{\qquad}c}
        \infer{\bindx f \in \LlGC}{f \in \LlG\;\;\;x \in Fv(f)} &
        \infer{f::c \in \LlGC}
                    {f \in \LlG\;\;\;c \in \LlGC\;\;\; Fv(f)\cap Fv(c) = \emptyset}
      \end{array}
      \\\\
    \begin{array}{c}
      \infer{\weak{x}{E} \in \LlG \cup \LlGC}
      {E \in \LlG \cup \LlGC \;\; x \notin Fv(E)}
   \end{array} \\\\
   \begin{array}{c}
      \infer{\cont{x}{x_1}{x_2}{E} \in \LlG \cup \LlGC}
      {E \in \LlG \cup \LlGC \;\; x_{1} \not = x_{2}\;\; x_1,x_2 \in Fv(E) \;\; x \notin Fv(E)\setminus \{x_{1},x_{2}\}}
    \end{array}
  \end{array}
$ }} \caption{$\LlG \cup \LlGC$:  $\llG$-expressions}
\label{fig:wf-gtz}
\end{figure}

The computation over the set of $\llG$-expressions reflects the
cut-elimination process. Four groups of reductions in
$\llG$-calculus are given in Figure~\ref{fig:red-llG}.

\begin{figure}[ht]
\centerline{ \framebox{ $
\begin{array}{rrcl}
(\beta_\mathsf{g}) & (\lambda x.t)(u::k) & \rightarrow & u(\bindx tk)\\
(\sigma) & T(\bindx v)              & \rightarrow & v\isub{T}{x}\\
(\pi) & (tk)k'                    & \rightarrow & t(\append{k}{k'})\\
(\mu) & \bindx xk             & \rightarrow & k\\[2mm]
(\gamma_1)       & \cont{x}{x_1}{x_2}{(\lambda y.t)} & \rightarrow
& \lambda y.\cont{x}{x_1}{x_2}{t}\\
(\gamma_2)       & \cont{x}{x_1}{x_2}{(tk)} & \rightarrow & (\cont{x}{x_1}{x_2}{t})k, \;\;\;\mbox{if} \; x_1,x_2 \notin Fv(k)\\
(\gamma_3)       & \cont{x}{x_1}{x_2}{(tk)} & \rightarrow & t(\cont{x}{x_1}{x_2}{k}), \;\;\;\mbox{if} \; x_1,x_2 \notin Fv(t)\\
(\gamma_4)       & \cont{x}{x_1}{x_2}{(\bindy t)} & \rightarrow & \bindy (\cont{x}{x_1}{x_2}{t})\\
(\gamma_5)       & \cont{x}{x_1}{x_2}{(t::k)} & \rightarrow & (\cont{x}{x_1}{x_2}{t})::k, \;\;\;\mbox{if} \; x_1,x_2 \notin Fv(k)\\
(\gamma_6)       & \cont{x}{x_1}{x_2}{(t::k)} & \rightarrow & t::(\cont{x}{x_1}{x_2}{k}), \;\;\;\mbox{if} \; x_1,x_2 \notin Fv(t)\\[2mm]
(\omega_1)       & \lambda x.(\weak{y}{t}) & \rightarrow & \weak{y}{(\lambda x.t)},\;\;\;x \neq y\\
(\omega_2)       & (\weak{x}{t})k & \rightarrow & \weak{x}{(tk)}\\
(\omega_3)       & t(\weak{x}{k}) & \rightarrow & \weak{x}{(tk)}\\
(\omega_4)       & \bindx (\weak{y}{t}) & \rightarrow & \weak{y}{(\bindx t)},\;\;\;x \neq y\\
(\omega_5)       & (\weak{x}{t})::k & \rightarrow & \weak{x}{(t::k)}\\
(\omega_6)       & t::(\weak{x}{k}) & \rightarrow & \weak{x}{(t::k)}\\[2mm]
(\gamma \omega_1)       & \cont{x}{x_1}{x_2}{(\weak{y}{e})} & \rightarrow &
\weak{y}{(\cont{x}{x_1}{x_2}{e})}  \qquad x_1\neq y \neq x_2\\
(\gamma \omega_2)       & \cont{x}{x_1}{x_2}{(\weak{x_1}{e})} & \rightarrow & e\isub{x}{x_{2}}\\
\end{array}
$ }} \caption{Reduction rules of $\llG$-calculus}
\label{fig:red-llG}
\end{figure}

The first group consists of $\beta_\mathsf{g}$, $\pi$, $\sigma$
and $\mu$ reductions from the $\lG$. New reductions are added to
deal with explicit contraction ($\gamma$ reductions) and weakening
($\omega$ reductions). The groups of $\gamma$ and $\omega$
reductions consist of rules that perform propagation of
contraction into the expression and extraction of weakening out of
the expression. This discipline allows us to optimize the
computation by delaying the duplication of terms on the one hand,
and by performing the erasure of terms as soon as possible on the
other. The equivalencies in $\llG$ are the ones given in
Figure~\ref{fig:equiv-rcl}, except for the fact that they refer to
$\llG$-expressions.

The meta-substitution $t\isub{u}{x}$ is defined as in
Figure~\ref{fig:sub-rcl} with the following additions:
$$
\begin{array}{rclcrcl}
(tk)\isub{u}{x}            & = & t\isub{u}{x}k, \;\; x \notin
Fv(k)
& & (tk)\isub{u}{x} & = & tk\isub{u}{x}, \;\; x \notin Fv(t) \\
(\bindy t)\isub{u}{x} & = & \bindy t\isub{u}{x} \\
(t::k)\isub{u}{x}        & = & t\isub{u}{x}::k, \;\; x \notin
Fv(k) & & (t::k)\isub{u}{x}     & = &t::k\isub{u}{x}, \;\; x
\notin Fv(t)
\end{array}
$$
In the $\pi$ rule, the meta-operator $@$, called
\emph{append}, joins two contexts and is defined as:
$$
\begin{array}{rclcrcl}
       \app{(\bindx t)}{k'} & = & \bindx tk'                        & \qquad & \app{(u::k)}{k'} & = & u::(\app{k}{k'})\\
\app{(\weak{x}{k})}{k'} & = & \weak{x}{(\app{k}{k'})} & \qquad & \app{(\cont{x}{y}{z}{k})}{k'} & = & \cont{x}{y}{z}{(\app{k}{k'})}.
\end{array}
$$

\subsection{Intersection types for $\llG$}
\label{subsec:llG-inttypes}

The type assignment system $\llG \cap$ that assigns strict types
to $\llG$-expressions is given in Figure~\ref{fig:inttyp-llG}. Due
to the sequent flavour of the $\llG$-calculus, here we distinguish
two sorts of type assignments:
\begin{itemize}
\item[-]
$\Gamma \vdash t:\tS$ for typing a term and
\item[-]
$\Gamma;\tB \vdash k:\tS$,
a type assignment with a \emph{stoup},
for typing a context.
\end{itemize}
A stoup is a place for the last formula in the antecedent, after
the semi-colon. The formula in the stoup is the place where
computation will continue.

The syntax of types and the related definitions are the same as in
$\rcl\cap$. The $\llG \cap$ system is also syntax-directed i.e.
the intersection is incorporated into already existing rules of
the simply-typed system. In the style of sequent calculus, left
intersection introduction is managed by the contraction rules
$(Cont_t)$ and $(Cont_k)$, whereas the right intersection
introduction is performed by the cut rule $(Cut)$ and left arrow
introduction rule $(\to_L)$. In these two rules $Dom(\Gamma_1)=
\ldots = Dom(\Gamma_n)$. The role of $\gztop$ has been already
explained in subsection~\ref{subsec:rcl-inttypes}.

\begin{figure}[ht]
\centerline{ \framebox{
$
\begin{array}{c}
\\
\infer[(Ax)]{x:\tS \vdash x:\tS}{}\\ \\
\infer[(\to_R)]{\Gamma \vdash \lambda x.t:\tA\to \tS}
                        {\Gamma,x:\tA \vdash t:\tS} \quad\quad
\infer[(Sel)]{\Gamma; \tA \vdash \bindx t:\tS}
                    {\Gamma, x:\tA \vdash t:\tS}\\ \\
\infer[(\to_L)]{\gztop \sqcap \Gamma_1 \sqcap ... \sqcap \Gamma_n,
\Delta;
            \cap_j^{m}(\cap_i^{n}\tS_i \to \tT_j) \vdash t::k:\tR}
             {\Gamma_0 \vdash t:\tS_0 & ... & \Gamma_n \vdash t:\tS_n & \Delta;\cap_j^{m}\tT_j \vdash k:\tR}\\ \\
\infer[(Cut)]{\gztop \sqcap \Gamma_1 \sqcap ... \sqcap \Gamma_n,
\Delta \vdash tk:\tT}
                {\Gamma_0 \vdash t:\tS_0 & ... & \Gamma_n \vdash t:\tS_n & \Delta; \cap_{i}^{n} \tS_i \vdash k:\tT}\\ \\
\infer[(Cont_t)]{\Gamma, z:\tA \cap \tB \vdash
\cont{z}{x}{y}{t}:\tS}
                    {\Gamma, x:\tA, y:\tB \vdash t:\tS} \quad\quad
\infer[(Weak_t)]{\Gamma, x:\top \vdash \weak{x}{t}:\tS}
                    {\Gamma \vdash t:\tS}\\ \\
\infer[(Cont_k)]{\Gamma, z:\tA \cap \tB; \tC \vdash
\cont{z}{x}{y}{k}:\tS}
                    {\Gamma, x:\tA, y:\tB; \tC \vdash k:\tS} \quad\quad
\infer[(Weak_k)]{\Gamma, x:\top; \tC \vdash \weak{x}{k}:\tS}
                    {\Gamma; \tC \vdash k:\tS}\\
\end{array}
$ }} \caption{$\llG \cap$: $\llG$-calculus with intersection types}
\label{fig:inttyp-llG}
\end{figure}

The Generation lemma induced by the proposed system is the
following:
\begin{lemma}[Generation lemma for $\llG\cap$]
\label{prop:intGL-Gtz} \rule{0in}{0in}
\begin{enumerate}
\item[(i)]\quad $\Gamma \vdash \lambda x.t:\tT\;\;$ iff there
exist $`a$ and $`s$ such that $\;\tT \equiv \tA\rightarrow
\tS\;\;$ and $\;\Gamma,x:\tA \vdash t:\tS.$
\item[(ii)]\quad $\Gamma;\tC \vdash t::k:\tR\;\;$ iff
$\;\Gamma={\Gamma'_0}^\top \sqcap \Gamma'_1 \sqcap ... \sqcap
\Gamma'_n, \Delta,\;$ $\tC \equiv \cap_j^{m} (\cap_i^{n}\tS_i \to
\tT_j)$, $\;\Delta;\cap_j^{m}\tT_j \vdash k:\tR\;$ and
$\;\Gamma'_l \vdash t:\tS_l$ for all $l \in \{0, \ldots, n\}$.
\item[(iii)]\quad $\Gamma \vdash tk:\tS\;\;$ iff
$\;\Gamma={\Gamma'_0}^\top \sqcap \Gamma'_1 \sqcap ... \sqcap
\Gamma'_n,\Delta,\;$ and there exist $\tT_j, j = 0, \ldots, n$
such that for all $j \in \{0, \ldots, n\}\;$ the following holds:
$\Gamma'_j \vdash t:\tT_j$, and $\;\Delta;\cap_{i}^{n} \tT_i
\vdash k:\tS.$
\item[(iv)]\quad $\Gamma;\tA \vdash \bindx t:\tS\;\;$ iff
$\;\Gamma,x:\tA \vdash t:\tS.$
\item[(v)]\quad $\Gamma \vdash \cont{z}{x}{y}{t}:\tS\;\;$ iff
there exist $`G', `a, `b$ such that $\;\Gamma=\Gamma', z:`a\cap`b$
and \\ $\;\Gamma', x:\tA, y:\tB \vdash t:\tS.$
\item[(vi)]\quad $\Gamma \vdash \weak{x}{t}:\tS\;\;$ iff
$\;\Gamma=\Gamma', x:\top$ and $\;\Gamma' \vdash t:\tS.$
\item[(vii)]\quad $\Gamma; \tC \vdash \cont{z}{x}{y}{k}:\tS\;\;$
iff there exist $`G',`a, `b$ such that $\;\Gamma=\Gamma', z: \tA\cap \tB$ and\\
$\;\Gamma', x:\tA, y:\tB; \tC \vdash k:\tS.$
\item[(viii)]\quad $\Gamma; \tC \vdash \weak{x}{k}:\tS\;\;$ iff
$\;\Gamma=\Gamma', x:\top$ and $\;\Gamma'; \tC \vdash k:\tS.$
\end{enumerate}
\end{lemma}

The proposed system satisfies the following properties.

\begin{lemma}
\label{prop:bases-weak}
\rule{0in}{0in}
\begin{itemize}
\item[(i)] If $\;\Gamma \vdash t:\tS\;$, then
$\;Dom(\Gamma)= Fv(t).$
\item[(ii)] If $\;\Gamma; \tA \vdash k:\tS\;$, then
$\;Dom(\Gamma)= Fv(k).$
\end{itemize}
\end{lemma}
\begin{proof}
Similar to the proof of Lemma~\ref{lem:dom-corr}.
\end{proof}

\begin{lemma}[Substitution lemma for $\llG \cap$]
\label{prop:sub-lemma-llG}
\rule{0in}{0in}
\begin{itemize}
\item[(i)] If $\;\Gamma, x:\cap_i^{n}\tT_i \vdash t:\tS\;$ and for
all $j=0,\ldots,n$, $\;\Delta_j \vdash u:\tT_j$, then\\
$\;\Gamma, \dztop \sqcap \Delta_1 \sqcap ... \sqcap \Delta_n
\vdash t\isub{u}{x}:\tS.$
\item[(ii)] If $\;\Gamma, x:\cap_i^{n} \tT_i;\tA \vdash k:\tS\;$
and for all $j=0,\ldots,n$, $\;\Delta_j \vdash u:\tT_j$,
then\\
 $\;\Gamma, \dztop \sqcap \Delta_1 \sqcap ... \sqcap
\Delta_n;\tA \vdash k\isub{u}{x}:\tS.$
\end{itemize}
\end{lemma}
\begin{proof}
By mutual induction on the structure of terms and contexts.
\end{proof}

\begin{proposition}[Append lemma]
\label{prop:app-lemma} \rule{0in}{0in} If $\;\Gamma_j;\tA \vdash
k:\tT_j\;$ for all $j=0,\ldots,n$, and $\;\Delta;\cap_i^{n} \tT_i
\vdash k':\tS$, then $\;\gztop \sqcap \Gamma_1 \sqcap \ldots
\sqcap \Gamma_n, \Delta;\tA \vdash \app{k}{k'}:\tS.$
\end{proposition}
\begin{proof}
By induction on the structure of the context $k$.
\end{proof}

\begin{proposition}[Subject equivalence for $\llG \cap$] \label{prop:subequiv-sequent}
\rule{0in}{0in}
\begin{itemize}
\item[(i)] For every $\llG$-term $t$: if $\;\Gamma \vdash t:\tS\;$
and $t \equiv_{\llG} t'$, then $\;\Gamma \vdash t':\tS.$
\item[(ii)] For every $\llG$-context $k$: if $\;\Gamma; \tA \vdash
k:\tS\;$ and $k \equiv_{\llG} k'$, then $\;\Gamma; \tA \vdash
k':\tS.$
\end{itemize}
\end{proposition}
\begin{proof}
By case analysis on the applied equivalence.
\end{proof}

\begin{proposition}[Subject reduction for $\llG\cap$]
\label{prop:sr-sequent}
\rule{0in}{0in}
\begin{itemize}
\item[(i)] For every $\llG$-term $t$: if $\;\Gamma \vdash
t:\tS\;$ and $t \to t'$, then $\;\Gamma \vdash t':\tS.$
\item[(ii)] For every $\llG$-context $k$: if $\;\Gamma; \tA
\vdash k:\tS\;$ and $k \to k'$, then $\;\Gamma; \tA \vdash
k':\tS.$
\end{itemize}
\end{proposition}
\begin{proof}
By case analysis on the applied reduction, using
Lemmas~\ref{prop:sub-lemma-llG} and \ref{prop:app-lemma} for the
cases of $(\sigma)$ and $(\pi)$ rule, respectively.
\end{proof}

\subsection{Typeability $\Rightarrow$ SN in $\llG \cap$}
\label{sec:typeSN_Gtz}

In this section, we prove the strong normalisation of the
$\llG$-calculus with intersection types. The termination is proved
by showing that the reduction on the set $\LlG\cup \LlGC$ of the
typeable $\llG$-expressions is included in a particular
well-founded relation, which we define as the lexicographic
product of three well-founded component relations. The first one
is based on the mapping of $\llG$-expressions into $\rcl$-terms.
We show that this mapping preserves types and that every
$\llG$-reduction can be
  simulated either by a $\rcl$-reduction or by an equality and each
  $\llG$-equivalence can be simulated by an $\rcl$-equivalence. The other
two well-founded orders are based on the introduction of quantities
designed to decrease a global measure associated with specific
$\llG$-expressions during the computation.

\begin{definition}
\label{def:mapping} The mapping
$\intt{\;\;}:\LlG\;\to\;\Rcl$ is
defined together with the auxiliary mapping
$\intc{\;\;}:\LlGC \; \to \;(\Rcl
\; \to \; \Rcl)$ in the following way:
$$
\begin{array}{lclclcl}
\intt{x} & = & x & & \intc{\bindx{t}}(M) & = & (\lambda x. \intt{t})M\\
\intt{\lambda x.t} & = & \lambda x. \intt{t} & & \intc{t::k}(M) & = & \intc{k}(M\intt{t})\\
\intt{\weak{x}{t}} & = & \weak{x}{\intt{t}} & & \intc{\weak{x}{k}}(M) & = & \weak{x}{\intc{k}(M)}\\
\intt{\cont {x}{y}{z}{t}} & = & \cont{x}{y}{z}{\intt{t}} & & \intc{\cont {x}{y}{z}{k}}(M) & = &\cont{x}{y}{z}{\intc{k}(M)} \\
\intt{tk} & = & \intc{k}(\intt{t})\\
\end{array}
$$
\end{definition}

\begin{lemma}
\rule{0in}{0in}
\begin{itemize}
\item[(i)]\label{lemma:pres-of-FV}
 $Fv(t) = Fv(\intt{t})$, for $t \in \LlG$.
\item[(ii)] \label{lemma:int-of-sub} $\intt{v\isub{t}{x}} =
\intt{v}\isub{\intt{t}}{x}$, for $v,t \in \LlG$.
\end{itemize}
\end{lemma}

We prove that the mappings $\intt{\;\;}$ and $\intc{\;\;}$
preserve types. In the sequel, the notation $\LGvdash{`G}{\tS}$
stands for $\{M \; \mid \; M \in \Rcl\; \&\; `G \vdash_{\rcl}
M:\tS\}$.

\begin{proposition}[Type preservation by $\intt{\;\;}$]\label{prop:soundness}
\rule{0in}{0in}
\begin{itemize}
\item[(i)] If $\;`G \vdash t:\tS$, then $`G \vdash_{\rcl}
\intt{t}:\tS$.
\item[(ii)] If $\;`G;\cap^n_i\tT_i \vdash k:\tS$, then
$\intc{k}:\LRvdash{\Delta_j}{\tT_j} \to \LRvdash{`G, \Delta
}{\tS}$, for all $j \in \{0, \ldots, n\}$ and for some $\Delta=
\dztop \sqcap \Delta_1 \sqcap ... \sqcap \Delta_n$.
\end{itemize}
\end{proposition}

\begin{proof}
The proposition is proved by simultaneous induction on
derivations. We distinguish cases according to the last typing
rule used.
\begin{itemize}
\item Cases $(Ax)$, $(\to_R)$, $(Weak_t)$ and $(Cont_t)$ are easy,
because the intersection type assignment system of $\rcl$ has
exactly the same rules.
\item Case $(Sel)$: the derivation ends
with the rule
$$
\infer[(Sel)]{`G; \tA \vdash \bindx t:\tS}
                    {`G, x:\tA \vdash t:\tS}
$$
By IH we have that $`G, x:\tA \vdash_{\rcl} \intt{t}:\tS$, where
$\tA=\cap_i^n \tT_i$. For any $M \in \Rcl$ such that $\Delta_j
\vdash_{\rcl} M:\tT_i$, for all $j \in \{0,\ldots,n\}$, we have
\begin{center}
\[ \prooftree
    \prooftree
        `G, x:\cap_i^n \tT_i \vdash_{\rcl} \intt{t}:\tS
        \justifies
        `G \vdash_{\rcl} \lambda x.\intt{t}:\cap_i^n \tT_i \to \tS
    \using{(\to_I)}
    \endprooftree
    \Delta_0 \vdash_{\rcl} M:\tT_0 \;\ldots\;\Delta_n \vdash_{\rcl} M:\tT_n
   \justifies
   `G, \dztop \sqcap \Delta_1 \sqcap \ldots \sqcap \Delta_n \vdash_{\rcl}(\lambda x.\intt{t})M:\tS
   \using{(\to_E)}
\endprooftree
\]
\end{center}
Since $(\lambda x.\intt{t})M=\intc{\bindx t}(M)$, we conclude that
$\intc{\bindx t}:\LGvdash{\Delta_j}{\tT_j} \to \LGvdash{`G, \dztop
\sqcap \Delta_1 \sqcap \ldots \sqcap \Delta_n}{\tS}$.
\item Case $(\to_L)$: the derivation ends with the rule
$$
\infer[(\to_L)]{ `G, \Delta; \cap^m_j (\cap^n_i\tS_i \to \tT_j)
\vdash t::k:\tR} {`G_0 \vdash t:\tS_0\;...\;`G_n \vdash t:\tS_n &
\Delta;\cap^m_j\tT_j \vdash k:\tR}
$$
for $`G=\gztop \sqcap `G_1 \sqcap \ldots \sqcap `G_n$. By IH we
have that $`G_l \vdash_{\rcl} \intt{t}:\tS_l$, for $l \in
\{0,\ldots,n\}$. For any $M \in \Rcl$ such that $`G'_j
\vdash_{\Rcl} M:\cap_{i}^{n}\tS_i \to \tT_j$, $j = 1, \ldots, m$
we have
$$
\infer[(\to_E)]{\gztop \sqcap `G_1 \sqcap \ldots \sqcap `G_n,`G'_j
\vdash_{\rcl} M \intt{t}:\tT_j}
                    {`G'_j \vdash_{\rcl} M:\cap_i^{n}\tS_i \to \tT_j \quad `G_0 \vdash_{\rcl} \intt{t}:\tS_0
                    \;\ldots\;`G_n \vdash_{\rcl} \intt{t}:\tS_n}
$$
From the right-hand side premise in the $(\to_L)$ rule, by IH, we
get that $\intc{k}$ is the function with the scope
$\intc{k}:\LRvdash{`G'''_j}{\tT_j} \to \LRvdash{`G''',
`G''}{\tR}$, for some $`G'''={\Gamma'''_0}^\top \sqcap `G'''_1
\sqcap ... \sqcap `G'''_n$. For $`G'''\equiv `G,`G'$ and by taking
$M\intt{t}$ as the argument of the function $\intc{k}$, we get
$`G, \Delta, `G'\vdash_{\rcl} \intc{k}(M\intt{t}):\tR$. Since
$\intc{k}(M\intt{t})=\intc{t::k}(M)$, we have that $`G, \Delta,
`G'\vdash_{\rcl} \intc{t::k}(M):\tR$. This holds for any
$M$ of the appropriate type, yielding\\
$\intc{t::k}:\LRvdash{`G'}{\cap^n_i\tS_i \to \tT_j} \to
\LRvdash{`G ,\Delta, `G'}{\tR}$, which is exactly what we need.
Case $(Cut)$: the derivation ends with the rule
$$
\infer[(Cut)]{\gztop \sqcap `G_1 \sqcap \ldots \sqcap `G_{n},
\Delta \vdash tk:\tS} {`G_{0} \vdash t:\tT_0 \ldots `G_{n} \vdash
t:\tT_n & \Delta; \cap \tT_i^{n} \vdash k:\tS}
$$
By IH we have that $`G_j \vdash_{\rcl} \intt{t}:\tT_j$ and
$\intc{k}:\LRvdash{`G_j'}{\tT_j} \to \LRvdash{`G', \Delta}{\tS}$
for all $j =0, \ldots, n$ and for $`G'=\gztop \sqcap `G'_1 \sqcap
\ldots \sqcap `G'_n$. Hence, for any $M \in \Rcl$ such that $`G_j'
\vdash_{\rcl} M:\tT_j$, $`G', \Delta \vdash_{\rcl}
\intc{k}(M):\tS$ holds. By taking $M \equiv \intt{t}$ and $`G'
\equiv `G$, we get $`G, \Delta \vdash_{\rcl}
\intc{k}(\intt{t}):\tS$. But $\intc{k}(\intt{t})=\intt{tk}$, so
the proof is done.
\item Case $(Weak_k)$: the derivation ends with the rule
$$
\infer[(Weak_k)]{`G, x:\top; \tB \vdash \weak{x}{k}:\tS}
                    {`G; \tB \vdash k:\tS}
$$
By IH we have that $\intc{k}$ is the function with the scope
$\intc{k}:\LRvdash{`G'_j}{\tT_j} \to \LRvdash{`G,{`G'_0}^\top
\sqcap `G'_1\sqcap \ldots \sqcap `G'_n }{ \tS}$, meaning that for
each $M \in \Rcl$ such that $`G'_j \vdash_{\rcl} M:\tT_j$ for all
$j \in \{0,\ldots,n\}$ holds ${`G'_0}^\top \sqcap `G'_1\sqcap
\ldots \sqcap `G'_n,`G \vdash_{\rcl} \intc{k}(M):\tS$. Now, we can
apply $(Weak)$ rule:
$$
\infer[(Weak)]{`G,{`G'_0}^\top \sqcap `G'_1\sqcap \ldots \sqcap
`G'_n, x:\top \vdash \weak{x}{\intc{k}(M)}:\tS}
                    {`G,{`G'_0}^\top
\sqcap `G'_1\sqcap \ldots \sqcap `G'_n \vdash \intc{k}(M):\tS}
$$
Since $\weak{x}{\intc{k}(M)}= \intc{\weak{x}{k}}(M)$, this means
that $\intc{\weak{x}{k}}:\LRvdash{`G'_j}{\tT_j} \to \LRvdash{`G,
{`G'_0}^\top \sqcap `G'_1\sqcap \ldots \sqcap `G'_n,x:\top}{\tS}$,
which is exactly what we wanted to get.
\item Case $(Cont_k)$: similar to the case $(Weak_k)$, relying on
the rule $(Cont)$ in $\rcl$.
\end{itemize}
\end{proof}

For the given encoding $\intt{\;\;}$, we show that each
$\llG$-reduction step can be simulated by an $\rcl$-reduction or
by an equality. In order to do so, we prove the following lemmas.
The proofs of Lemma~\ref{lemma:pi2-for-contexts} and
Lemma~\ref{lemma:int-of-append}, according to
\cite{espighilivet07}, use Regnier's $\sigma$ reductions,
investigated in \cite{regn94}.

$$\begin{array}{rcl}
((\lambda x. M)N)P & \to & (\lambda x. (MN)) P \; \;  x \notin P \\
(\lambda x y. M)N & \to & \lambda y.((\lambda x.M)N) \; \; y \notin N \\
M ((\lambda x.P)N) & \to & (\lambda x. MP)N \; \;  x \notin M
\end{array}
$$

\begin{lemma}\label{lemma:context-closure}
If $M \to_{\rcl} M'$, then $\intc{k}(M) \to_{\rcl} \intc{k}(M').$
\end{lemma}

\begin{lemma}\label{lemma:pi2-for-contexts}
$\intc{k}((\lambda x.P)N) \to_{\rcl} (\lambda x.\intc{k}(P))N.$
\end{lemma}

\begin{lemma}\label{lemma:int-of-append}
If $M \in \Lambda^{\rc}$ and $k,k' \in\LlGC$, then
$\intc{k'}\circ\intc{k}(M) \to_{\rcl} \intc{\app{k}{k'}}(M).$
\end{lemma}

\begin{lemma}
\rule{0in}{0in}
\begin{itemize}
\item[(i)] \label{lemma:prop-of-isub} If $x \notin Fv(k)$,  then
$(\intc{k}(M))\isub{N}{x} = \intc{k}(M\isub{N}{x}).$
\item[(ii)]\label{lemma:prop-of-cont} If $x,y \notin Fv(k)$,  then
$\cont{z}{x}{y}{(\intc{k}(M))} \to_{\rcl}
\intc{k}(\cont{z}{x}{y}{M}).$
\item[(iii)]\label{lemma:push-of-weak} $\intc{k}(\weak{x}{M})
\to_{\rcl} \weak{x}{\intc{k}(M)}.$
\end{itemize}
\end{lemma}

Now we can prove that the reduction rules of $\llG$ can be
simulated by the reduction rules or an equality in the $\rcl$-calculus.
Moreover, the equivalences of $\llG$-calculus are preserved in
$\rcl$-calculus.

\begin{theorem}[Simulation of \ensuremath{\gcw}-reduction by
\ensuremath{\rcl}-reduction]
\label{prop:simulation}
\rule{0in}{0in}
\begin{itemize}
\item[(i)] If a term $t \to_{\llG} t'$, then $\intt{t} \to_{\rcl}
\intt{t'}$.
\item[(ii)] If a context $k \to_{\llG} k'$ by $\gamma_6$ or
$\omega_6$ reduction, then $\intc{k}(M) \equiv \intc{k'}(M)$, for
any $M \in \Lambda^{\rc}$.
\item[(iii)] If a context $k \to_{\llG} k'$ by some other
reduction, then $\intc{k}(M) \to_{\rcl} \intc{k'}(M)$, for any $M
\in \Lambda^{\rc}$.
\item[(iv)] If $t \equiv_{\llG} t'$, then $\intt{t} \equiv_{\rcl}
\intt{t'}$, and if $k \equiv_{\llG} k'$, then $\intc{k}(M)
\equiv_{\rcl} \intc{k'}(M)$, for any $M \in \Rcl$.
\end{itemize}
\end{theorem}
\begin{proof}
The proof goes by case analysis on the outermost reduction or
equivalence performed, using  Definition~\ref{def:mapping}.
\end{proof}

The previous proposition shows that $\beta_\mathsf{g}$, $\pi$,
$\sigma$, $\mu$, $\gamma_1$ - $\gamma_5$, $\omega_1$ - $\omega_5$,
$\gamma\omega_1$ and $\gamma\omega_2$ $\llG$-reductions are
interpreted by $\rcl$-reductions and that $\gamma_6$ and
$\omega_6$ $\llG$-reductions are interpreted by an identity in the
$\rcl$. Since the set of equivalences of the two calculi coincide,
they are trivially preserved. If one wants to prove that there is
no infinite sequence of $\llG$-reductions one has to prove that
there cannot exist an infinite sequence of $\llG$-reductions which
are all interpreted as equalities. To prove this, one shows that
if a term is reduced with such a $\llG$-reduction, it is reduced
for another order that forbids infinite decreasing chains. This
order is itself composed of several orders, free of infinite
decreasing chains (Definition~\ref{def:ord}).

\begin{definition}
\label{def:x} \label{def:cnorm}\label{def:wnorm}
The functions
$\size{\:},\;\cnorm{\;},\;\wnorm{\;}: \LlG \to \mathbb{N}$ are
defined as follows:
$$
\begin{array}{rclrcl}
\size{x} & = & 1 & \size{tk} & = & \size{t} + \size{k} \\
\size{\lambda x.t} & = & 1 + \size{t} & \size{\bindx{t}} & = & 1 + \size{t} \\
\size{\weak{x}{e}} & = & 1 + \size{e}  & \size{t::k} & = & \size{t} + \size{k} \\
\size{\cont {x}{y}{z}{e}} & = & 1 + \size{e} & & &
\end{array}
$$

$$
\begin{array}{rclcrcl}
\cnorm{x} & = & 0 & \qquad & \wnorm{x} & = & 1\\
\cnorm{\lambda x.t} & = & \cnorm{t} & \qquad & \wnorm{\lambda x.t} & = & 1 + \wnorm{t}\\
\cnorm{\weak{x}{e}} & = & \cnorm{e} & \qquad & \wnorm{\weak{x}{e}} & = & 0\\
\cnorm{\cont{x}{y}{z}{e}} & = & \cnorm{e}  + \size{e} & \qquad &
\wnorm{\cont {x}{y}{z}{e}} & = & 1 + \wnorm{e}\\
\cnorm{tk} & = & \cnorm{t} + \cnorm{k} & \qquad & \wnorm{tk} & = & 1 + \wnorm{t} + \wnorm{k}\\
\cnorm{\bindx{t}} & = & \cnorm{t} & \qquad & \wnorm{\bindx{t}} & = & 1 + \wnorm{t}\\
\cnorm{t::k} & = & \cnorm{t} + \cnorm{k} & \qquad & \wnorm{t::k} & = & 1 + \wnorm{t} + \wnorm{k}
\end{array}
$$
\end{definition}

\begin{lemma}\label{lemma:cnorm}
For all $e, e' \in \LlG$:
\begin{itemize}
 \item[(i)] If $e \; \to_{\gamma_6} \; e'$, then $\cnorm{e} >
 \cnorm{e'}$.
 \item[(ii)] If $e \; \to_{\omega_6} \; e'$, then
 $\cnorm{e} = \cnorm{e'}$.
 \item[(iii)] If $e \; \equiv_{\llG} \; e'$, then
 $\cnorm{e} = \cnorm{e'}$.
\end{itemize}
\end{lemma}

\begin{lemma}\label{lemma:wnorm}
\rule{0in}{0in}
\begin{itemize}
 \item[(i)] For all $e, e' \in \LlG$: If $e \; \to_{\omega_6} \; e'$, then
 $\wnorm{e} > \wnorm{e'}$.
 \item[(ii)]  If $e \; \equiv_{\llG} \; e'$, then
 $\wnorm{e} = \wnorm{e'}$.
\end{itemize}
\end{lemma}

Now we can define the following orders based on the previously
introduced mapping and norms.
\begin{definition}\label{def:ord}
We define the following strict orders and equivalencies on $\LlG
\cap$:
\begin{itemize}
\item[(i)] $t >_{\rcl} t'$\; iff \; $\intt{t} \rightarrow^+_{\rcl} \intt{t'}$;\;\;  $t =_{\rcl} t'$\; iff \;
            $\intt{t} \equiv_{\rcl} \intt{t'}$\\
          $k >_{\rcl} k'$\; iff \; $\intc{k}(M) \rightarrow^+_{\rcl} \intt{k'}(M)$\; for every $\rcl$ term $M$ ;\\
          $k =_{\rcl} k'$\; iff \; $\intc{k}(M) \equiv_{\rcl} \intc{k'}(M)$ or $\intc{k}(M) \equiv \intt{k'}(M)$ for every $\rcl$ term $M$;
    \item[(ii)]  $e >_c e'$\; iff \; $\cnorm{e} > \cnorm{e'}$;\;\; $e
    =_c e'$\; iff \;$\cnorm{e} = \cnorm{e'};$
    \item[(iii)] $e >_w e'$\; iff \;$\wnorm{e} > \wnorm{e'}$;\;\; $e =_w e'$\; iff \;
    $\wnorm{e} = \wnorm{e'};$
\end{itemize}
\end{definition}

The lexicographic product of two orders $>_1$ and $>_2$ is
 defined as~\cite{baadnipk98}:\\
\[a >_1 \times_{lex} >_2 b \; \Leftrightarrow \;a
>_1 b \;\;or\;\; (a =_1 b \;and \; a >_2 b). \]

\begin{definition}
We define the relation $\gg$ on  $\LlG$ as the lexicographic product:
\[ \gg \;\;\;= \;\;\;>_{\rcl}\; \times_{lex} \; >_c \; \times_{lex} \;  >_w.\]
\end{definition}

The following propositions proves that the reduction relation on
the set of typed $\llG$-expressions is included in the given
lexicographic product $\gg$.
\begin{proposition}\label{prop:inclusion}
For each $e`: \LlG$: if  $e \to e'$, then $e \gg
e'$.
\end{proposition}
\begin{proof}
By case analysis on the kind of reduction and the structure of $\gg$.\\
If $e \to e'$ by $\beta_\mathsf{g}$,  $\sigma$, $\pi$, $\mu$,
$\gamma_1$, $\gamma_2$, $\gamma_3$, $\gamma_4$ $\gamma_5$,
$\gamma\omega_1$, $\gamma\omega_2$, $\omega_1$, $\omega_2$,
$\omega_3$ $\omega_4$ or $\omega_5$ reduction, then $e >_{\rcl}
e'$
by Proposition~\ref{prop:simulation}.\\
If $e \to e'$ by $\gamma_6$, then $e =_{\rcl} e'$ by
Proposition~\ref{prop:simulation}, and $e >_c e'$ by Lemma~\ref{lemma:cnorm}.\\
Finally, if $e \to e'$ by $\omega_6$, then $e =_{\rcl} e'$ by
Proposition~\ref{prop:simulation}, $e =_c e'$ by
Lemma~\ref{lemma:cnorm} and $e >_w e'$ by
Lemma~\ref{lemma:wnorm}.
\end{proof}

Strong normalisation of $"->"$ is another terminology for the
well-foundness of the relation $"->"$ and it is well-known that a
relation included in a well-founded relation is well-founded and
that the lexicographic product of well-founded relations is
well-founded.

\begin{theorem}[Strong normalisation of the $\llG \cap$]\label{prop:SN-gtz}
Each expression in $\LlG \cap$ is strongly normalising.
\end{theorem}

\begin{proof}
  The reduction $"->"$ is well-founded on $\LlG\cap$  as it is
  included (Proposition~\ref{prop:inclusion}) in the relation $\gg$ which is well-founded as the lexicographic
  product of the well-founded relations $>_{\rcl}$, $>_c$ and
  $>_w$. Relation $>_{\rcl}$ is based on the interpretation $\intt{~}:\LlG
"->" \Rcl$.  By Proposition~\ref{prop:soundness} typeability is
preserved by the interpretation $\intt{~}$ and $"->"_{\rcl}$ is strongly normalising
(i.e., well-founded) on $\Rcl\cap$
(Section~\ref{sec:reducibility}), hence $>_{\rcl}$ is well-founded
on $\LlG\cap$.
  Similarly, $>_c$ and $>_w$ are well-founded, as they are based on
  interpretations into the well-founded relation $>$ on the set~$\mathbb{N}$ of natural
  numbers.
\end{proof}

\subsection{SN $\Rightarrow$ Typeability in $\llG \cap$}
\label{sec:SNtype_Gtz}

Now, we want to prove that if a $\llG$-term is SN, then it is
typeable in the system $\llG \cap$. We follow the procedure used
in Section~\ref{sec:SNtype}. The proofs are similar to the ones in
Section~\ref{sec:SNtype}.

The abstract syntax of $\llG$-normal forms is the following:

\hspace*{5mm}
$
\begin{array}{rcl}
\hspace*{-3mm} t_{nf} & ::= & x\,|\,\lambda
x.t_{nf}\,|\,\lambda x. \weak{x}{t_{nf}}\,|\,x(t_{nf}::k_{nf})
\,|\, \cont{x}{y}{z}{y(t_{nf}::k_{nf})}\\
\hspace*{-3mm} k_{nf} & ::= & \bindx t_{nf}\,|\,\bindx
\weak{x}{t_{nf}}\,|\,t_{nf}::k_{nf}
\,|\, \cont{x}{y}{z}{(t_{nf}::k_{nf})},\;y \in Fv(t_{nf}), z \in Fv(k_{nf})\\
\hspace*{-3mm} w_{nf} & ::= & \weak{x}{e_{nf}}\,|\,\weak{x}{w_{nf}}\\
\end{array}
$

We use $e_{nf}$ for any $\llG$-expression in the normal form.
\begin{proposition}\label{prop:nf-are-typ-Gtz}
$\llG$-normal forms are typeable in the system $\llG \cap$.
\end{proposition}
\begin{proof}
By mutual induction on the structure of $t_{nf}$, $k_{nf}$ and
$w_{nf}$.
\end{proof}

\noindent The following two lemmas explain the behavior of the
meta operators $\isub{\;}{\;}$ and $\app\;$  during expansion.

\begin{lemma}[Inverse substitution lemma]\label{prop:inv-subst-Gtz}
\rule{0in}{0in}
\begin{itemize}
\item[(i)] Let $\;\Gamma \vdash t\isub{u}{x}:\tS\;$ and $u$
typeable. Then, there exist $\Delta_j$ and $\tT_j, \;j=0,\ldots,
n$  such that $\Delta_j \vdash u:\tT_j\;$ and $\Gamma',
x:\cap_{i}^{n}\tT_i \vdash t:\tS$, where $\Gamma = \Gamma',\dztop
\sqcap \Delta_1 \sqcap \ldots \sqcap \Delta_n$.
\item[(ii)] Let $\;\Gamma;\tC \vdash k\isub{u}{x}:\tS\;$ and $u$
typeable. Then, there are $\Delta_j$ and $\tT_j, \;j=0,\ldots, n$
such that $\Delta_j \vdash u:\tT_j$ and $\Gamma',
x:\cap_{i}^{n}\tT_i;\tC \vdash k:\tS$, where $\Gamma =
\Gamma',\dztop\sqcap\Delta_1\sqcap\ldots\sqcap\Delta_n$.
\end{itemize}
\end{lemma}
\begin{proof}
By mutual induction on the structure of terms and contexts.
\end{proof}

\begin{lemma}[Inverse append lemma]
\label{prop:inverseapp-Gtz} If $~~\Gamma;\tA\vdash
\app{k}{k'}:\tS$, then there are $\Delta_j$ and $\tT_j,
\;j=0,\ldots, n$ such that $\Delta_j;\tA\vdash k:\tT_j\;$ and
$\Gamma';\cap_i^n\tT_i\vdash k':\tS$, where $\Gamma=\Gamma',\dztop
\sqcap \Delta_1 \sqcap \ldots \sqcap \Delta_n$.
\end{lemma}

\begin{proof}
By induction on the structure of the context $k$.
\end{proof}

Now we prove that the type of a term is preserved during the
expansion.
\begin{proposition}[Head subject expansion]
\label{prop:sub-exp-Gtz} For every $\llG$-term $t$: if $t \to t'$,
$t$ is contracted redex and $\;\Gamma \vdash t':\tS\;$, then
$\;\Gamma \vdash t:\tS$.
\end{proposition}
\begin{proof}
By case study according to the applied reduction.
\end{proof}

\begin{theorem}[SN $\Rightarrow$ typeability]\label{thm:SNtypable-Gtz}
All strongly normalising $\llG$ terms are typeable in the
$\llG\cap$ system.
\end{theorem}
\begin{proof}
Analogous to the proof of Theorem~\ref{thm:SNtypable}.
\end{proof}

Now we give a characterisation of strong normalisation in $\llG$-calculus.

\begin{theorem}
In $\llG$-calculus, the term $t$ is strongly normalising if and
only if it is typeable in $\llG\cap$.
\end{theorem}

\begin{proof}
Immediate consequence of Theorems~\ref{prop:SN-gtz} and~\ref{thm:SNtypable-Gtz}.
\end{proof}

\section{Intersection types for the resource control lambda calculus with explicit substitution $\rclx$}
\label{sec:lambda_xsub}

\subsection{Resource control lambda calculus with explicit substitution $\rclx$}
\label{subsec:rclx}

The \emph{resource control} lambda calculus with explicit
substitution $\rclx$, is an extension of the
$\lambda_{\mathsf{x}}$-calculus with explicit operators for
weakening and contraction. It corresponds to the $\llxr$-calculus
of Kesner and Lengrand, proposed in \cite{kesnleng07}, and also
represents a vertex of ``the prismoid of resources''.

The {\em pre-terms} of $\rclx$-calculus are given by the following
abstract syntax:
$$
\begin{array}{lcrcl}
\textrm{Pre-terms}    &  & \f & ::= & x \,|\,\lambda x.\f \,|\, \f
\f\,|\,\f\xsub{\f}{x}\,|\,\weak{x}{\f} \,|\,
\cont{x}{x_1}{x_2}{\f}
\end{array}
$$
The only point of difference with respect to $\rcl$-calculus is
the operator of explicit substitution $\xsub{\;}{\;}$.

The set of free variables of a pre-term $\f$, denoted by $Fv(\f)$,
is defined as follows:

\hspace*{10mm}$\begin{array}{c} Fv(x) = x; \quad Fv(\lambda x.\f)
= Fv(\f)\setminus\{x\};\\
Fv(\f \g) = Fv(\f) \cup Fv(\g);\quad Fv(\f \xsub{\g}{x}) = (Fv(\f)\setminus \{x\}) \cup Fv(\g)\\
Fv(\weak{x}{\f}) = \{x\} \cup Fv(\f); \quad
Fv(\cont{x}{x_1}{x_2}{\f}) = \{x\} \cup Fv(\f)\setminus
\{x_1,x_2\}.
\end{array}$

In $\f \xsub{\g}{x}$, the substitution binds the variable $x$ in
$\f$.

The set of $\rclx$-{\em terms}, denoted by $\Rclx$ and ranged over
by $M,N,P,M_1,...$. is a subset of the set of pre-terms, defined by the rules
in Figure~\ref{fig:xsubterms}.
\begin{figure}[htpb]
\centerline{ \framebox{ $
    \begin{array}{c}
      \begin{array}{c@{\qquad\qquad}c}
        \infer{x \in \Rclx}{} &
        \infer{\lambda x.\f \in \Rclx}
        {\f \in \Rclx \;\; x \in Fv(\f)}
      \end{array}
      \\\\
    \infer{\f \g \in \Rclx}
    {\f \in \Rclx\;\; \g \in \Rclx \;\; Fv(\f) \cap Fv(\g) = \emptyset}
    \\\\
    \infer{\f \xsub{\g}{x} \in \Rclx}
    {\f \in \Rclx\;\; \g \in \Rclx \;\; x \in Fv(\f) \;\; (Fv(\f)\setminus\{x\}) \cap Fv(\g) = \emptyset}
    \\\\
    \begin{array}{c@{\qquad}c}
      \infer{\weak{x}{\f} \in \Rclx}
      {\f \in \Rclx \;\; x \notin Fv(\f)}&
      \infer{\cont{x}{x_1}{x_2}{\f} \in \Rclx}
      {\f \in \Rclx \;\; x_{1} \not = x_{2},\;\; x_1,x_2 \in Fv(\f) \;\; x \notin Fv(\f) \setminus \{x_{1},x_{2}\}}
    \end{array}
  \end{array}
$ }} \caption{$\Rclx$:  $\rclx$-terms} \label{fig:xsubterms}
\end{figure}

The notion of terms corresponds to the notion of linear terms in
\cite{kesnleng07}.

The reduction rules of $\rclx$-calculus are presented in
Figure~\ref{fig:red-rclx}.
\begin{figure}[htbp]
\centerline{ \framebox{ $
\begin{array}{lrcl}
(\beta_\mathsf{x})             & (\lambda x.M)N & \rightarrow & M\xsub{N}{x} \\[2mm]
(\sigma_1)          & x \xsub{N}{x} & \rightarrow & N\\
(\sigma_2)          &  (\lambda y.M)\xsub{N}{x} & \rightarrow & \lambda y.M \xsub{N}{x}\\
(\sigma_3)          & (MP) \xsub{N}{x} & \rightarrow & M \xsub{N}{x} P,\;\mbox{if} \; x \notin Fv(P)\\
(\sigma_4)          & (MP) \xsub{N}{x} & \rightarrow & M P \xsub{N}{x},\;\mbox{if} \; x \notin Fv(M)\\
(\sigma_5)          & (\weak{x}{M}) \xsub{N}{x} & \rightarrow & \weak{Fv(N)}{M}\\
(\sigma_6)          & (\weak{y}{M}) \xsub{N}{x} & \rightarrow & \weak{y}{M\xsub{N}{x}},\;\mbox{if} \; x \neq y\\
(\sigma_7)          & (\cont{x}{x_1}{x_2}{M})\xsub{N}{x} & \rightarrow &
                    \cont{Fv(N)}{Fv(N_1)}{Fv(N_2)}{M\xsub{N_1}{x_1}\xsub{N_2}{x_2}}\\
(\sigma_8)          & (M \xsub{N}{x})\xsub{P}{y} & \rightarrow & M \xsub{N\xsub{P}{y}}{x},\;\mbox{if} \; y \notin Fv(M) \setminus \{x\} \\[2mm]
(\gamma_1)          & \cont{x}{x_1}{x_2}{(\lambda y.M)} &
\rightarrow & \lambda y.\cont{x}{x_1}{x_2}{M}\\
(\gamma_2)          & \cont{x}{x_1}{x_2}{(MN)} & \rightarrow &
(\cont{x}{x_1}{x_2}{M})N, \;\mbox{if} \; x_1,x_2 \not\in Fv(N)\\
(\gamma_3)          & \cont{x}{x_1}{x_2}{(MN)} & \rightarrow &
M(\cont{x}{x_1}{x_2}{N}), \;\mbox{if} \; x_1,x_2 \not\in Fv(M)\\
(\gamma_4)          & \cont{x}{x_1}{x_2}{(M\xsub{N}{y}) &
\rightarrow &
M\xsub{\cont{x}{x_1}{x_2}{N}}{y}, \;\mbox{if} \; x_1,x_2 \notin Fv(M)\setminus \{y\}\\[2mm]
(\omega_1)          & \lambda x.(\weak{y}{M}) & \rightarrow & \weak{y}{(\lambda x.M)},\;x \neq y\\
(\omega_2)          & (\weak{x}{M})N & \rightarrow & \weak{x}{(MN)}\\
(\omega_3)          & M(\weak{x}{N}) & \rightarrow & \weak{x}{(MN)}\\
(\omega_4)          & M\xsub{\weak{x}{N}}{y} & \rightarrow & \weak{x}{(M\xsub{N}{y})}\\[2mm]
(\gamma \omega_1)   & \cont{x}{x_1}{x_2}{(\weak{y}{M})} &
\rightarrow & \weak{y}{(\cont{x}{x_1}{x_2}{M})},\;y \neq x_1,x_2\\
(\gamma \omega_2)   & \cont{x}{x_1}{x_2}{(\weak{x_1}{M})} &
\rightarrow & M \xsub{x}{x_2}
\end{array}
$ }}}\caption{Reduction rules of $\rclx$-calculus}
\label{fig:red-rclx}
\end{figure}

In the $\rclx$, one works modulo equivalencies given in
Figure~\ref{fig:equiv-rclx}.

\begin{figure}
\centerline{ \framebox{ $
\begin{array}{lrcl}
(`e_1) & \weak{x}{(\weak{y}{M})} & \equiv_{\rclx} &
\weak{y}{(\weak{x}{M})}\\
(`e_2) & \cont{x}{x_1}{x_2}{M} & \equiv_{\rclx} & \cont{x}{x_2}{x_1}{M}\\
(`e_3) & \cont{x}{y}{z}{(\cont{y}{u}{v}{M})} & \equiv_{\rclx}  &
\cont{x}{y}{u}{(\cont{y}{z}{v}{M})}\\
(`e_4) & \cont{x}{x_1}{x_2}{(\cont{y}{y_1}{y_2}{M})} &
\equiv_{\rclx}  &
\cont{y}{y_1}{y_2}{(\cont{x}{x_1}{x_2}{M})},\;\; x \neq y_1,y_2, \; y \neq x_1,x_2 \\
(`e_5) & M\xsub{N}{x}\xsub{P}{y} & \equiv_{\rclx}  &
M\xsub{P}{y}\xsub{N}{x},\;\; x \notin Fv(P),\; y \notin Fv(M) \\
(`e_6) & (\cont{y}{y_1}{y_2}{M}) \xsub{N}{x} & \equiv_{\rclx} &
\cont{y}{y_1}{y_2}{M\xsub{N}{x}},\;\;x \neq y,\; y_1,y_2 \notin
Fv(N)
\end{array}
$ }} \caption{Equivalences in $\rclx$-calculus}
\label{fig:equiv-rclx}
\end{figure}

\subsection{Intersection types for $\rclx$}
\label{subsec:rclx-inttypes}

In this subsection we introduce intersection type assignment
system which assigns \emph{strict types} to $\rclx$-terms.
The system is syntax-directed, hence significantly different
from the one proposed in \cite{lenglescdougdezabake04}.

The syntax of types and the definitions of type assignment, basis,
etc.\ are the same as in the case of the system $\rcl \cap$. The
type assignment system $\rclx \cap$ is given in
Figure~\ref{fig:typ-rclx-int}. The only difference with respect to
the $\rcl \cap$ is the presence of one new type assignment rule,
namely $(Subst)$ for typing the explicit substitution. The rules
$(\to_E)$ and $(Subst)$ are constructed in the same manner, as
explained in subsection~\ref{subsec:rcl-inttypes}.

\begin{figure}[ht]
\centerline{ \framebox{ $
\begin{array}{c}
\\
\infer[(Ax)]{x:\tS \vdash x:\tS}{}\\ \\
\infer[(\to_I)]{\Gamma \vdash \lambda x.M:\tA \to \tS}
                        {\Gamma,x:\tA \vdash M:\tS} \quad\quad
\infer[(\to_E)]{\Gamma, \dztop \sqcap \Delta_1 \sqcap ... \sqcap
\Delta_n \vdash MN:\tS}
                    {\Gamma \vdash M:\cap^n_i \tT_i \to \tS & \Delta_0 \vdash N:\tT_0 & ... & \Delta_n \vdash N:\tT_n}\\ \\
\infer[(Subst)]{\Gamma, \dztop \sqcap \Delta_1 \sqcap ... \sqcap
\Delta_n \vdash M\xsub{N}{x}:\tS}
                    {\Gamma, x:\cap^n_i \tT_i \vdash M:\tS & \Delta_0 \vdash N:\tT_0 & ... & \Delta_n \vdash N:\tT_n}\\ \\
\infer[(Cont)]{\Gamma, z:\tA \cap \tB \vdash
\cont{z}{x}{y}{M}:\tS}
                    {\Gamma, x:\tA, y:\tB \vdash M:\tS} \quad\quad
\infer[(Weak)]{\Gamma, x:\top \vdash \weak{x}{M}:\tS}
                    {\Gamma \vdash M:\tS}\\
\end{array}
$ }} \caption{$\rclx \cap$: $\rclx$-calculus with intersection
types} \label{fig:typ-rclx-int}
\end{figure}

\begin{proposition}[Generation lemma for $\rclx \cap$]
\label{prop:intGL_xsub} \rule{0in}{0in}
\begin{enumerate}
\item[(i)] $\Gamma \vdash \lambda x.M:\tT\;\;$ iff there exist
$`a$ and $`s$ such that $\;\tT\equiv \tA\rightarrow \tS\;\;$ and
$\;\Gamma,x:\tA \vdash M:\tS.$
\item[(ii)] $\Gamma \vdash MN:\tS\;\;$ iff there exist $\Delta_j$
and $\tT_j,\;j  = 0, \ldots, n$ such that $\;\Delta_{j} \vdash
N:\tT_j$ and $\Gamma' \vdash M:\cap_{i}^{n}\tT_i\to \tS$, moreover
$\;\Gamma=\Gamma',\dztop \sqcap \Delta_1 \sqcap \ldots \sqcap
\Delta_{n}$.
\item[(iii)] $\Gamma \vdash M\xsub{N}{x}:\tS\;\;$ iff there exist
a type  $\tA =\cap^{n}_{j=0} \tT_j,\;$ such that for all $j \in
\{0, \ldots, n\}$, $\;\Delta_{j} \vdash N:\tT_j$ and
$\Gamma',x:\cap_{i}^{n}\tT_i \vdash M:\tS$, moreover
$\Gamma=\Gamma',x:\tA,\dztop \sqcap \Delta_1 \sqcap \ldots \sqcap
\Delta_{n}$.
\item[(iv)] $\Gamma \vdash \cont{z}{x}{y}{M}:\tS\;\;$ iff there
exist $`G', `a, `b$ such that
$\;\Gamma=\Gamma',z:`a\cap`b$ \\
and $\;\Gamma', x:\tA, y:\tB \vdash M:\tS.$
\item[(v)] $\Gamma \vdash \weak{x}{M}:\tS\;\;$ iff
 $\;\Gamma=\Gamma', x:\top$ and $\;\Gamma' \vdash M:\tS.$
\end{enumerate}
\end{proposition}

The proposed system also satisfies preservation of free variables,
bases intersection and subject reduction and equivalence.



\section{Conclusions}
\label{sec:conclusion}

In this paper, we have proposed intersection type assignment systems for:
\begin{itemize}
\item
resource control lambda calculus $\rcl$, which corresponds to $\lambda_{{\tt C W}}$ of \cite{kesnrena09};
\item
resource control sequent lambda calculus $\llG$ of~\cite{ghilivetlesczuni11} and
\item
resource control calculus with explicit substitution $\rclx$ of~\cite{kesnleng07}.
\end{itemize}
The three intersection type assignment systems proposed here give
a complete characterization of strongly normalizing terms for
these three calculi. The strong normalisation of typeable resource
control lambda terms is proved directly by an appropriate
modification of the reducibility method, whereas the same property
for resource control sequent lambda terms is proved by
well-founded lexicographic order based on suitable embedding into
the former calculus and the strong normalisation of the calculus
with explicit substitution is given by its interpretation in the
resource control lambda calculus. This paper expands the range of
the intersection type techniques and combines different methods in
the strict types environment.  It should be noticed that the
strict control on the way variables are introduced determines the
way terms are typed in a given environment. Basically, in a given
environment no irrelevant intersection types are introduced. The
flexibility on the choice of a type for a term, as it is used in
rule $(\to_E)$ in Figure~\ref{fig:typ-rcl-int}, comes essentially
from the choice one has in invoking the axiom.  Unlike the
approach of introducing non-idempotent intersection types into the
calculus with some kind of resource management~\cite{pagaronc10},
our intersection is idempotent. As a consequence, our type
assignment system corresponds to full intuitionistic logic, while
non-idempotent intersection type assignment systems correspond to
intuitionistic linear logic.

The three presented calculi $\rcl$, $\llG$ and $\rclx$ are good
candidates to investigate the computational content of
substructural logics~\cite{schrdose93}, both in natural deduction
and sequent calculus.  The motivation for these logics comes from
philosophy (Relevant Logics), linguistics (Lambek Calculus) to
computing (Linear Logic). Since the basic idea of resource control
is to explicitly handle structural rules, the control operators
could be used to handle the absence of (some) structural rules in
substructural logics such as weakening, contraction,
commutativity, associativity. This would be an interesting
direction for further research. Another direction will involve the
investigation of the use of intersection types, being a powerful
means for building models of lambda
calculus~\cite{barecoppdeza83,dezaghillika04}, in constructing
models for sequent lambda calculi. Finally, one may wonder how the
strict control on the duplication and the erasure of variables
influences the type reconstruction of
terms~\cite{DBLP:journals/entcs/BoudolZ05,kfouwell04}.

\textbf{Acknowledgements:} We would like to thank the ICTAC 2011
anonymous referees for their careful reading and many valuable
comments, which helped us improve the final version of the paper.
We would also like to thank Dragi\v sa \v Zuni\' c for
participating in the earlier stages of the work.

\bibliographystyle{plain}

\end{document}